\documentclass[12pt,reqno]{amsart}
\usepackage{amsmath,amssymb,amsthm,mathtools,calc,verbatim,enumitem,tikz,url,hyperref,mathrsfs,cite,fullpage}
\usepackage{bbm}
\usepackage{stmaryrd}
\usepackage{textcomp}
\usepackage{setspace}
\usepackage{amssymb}
\usepackage{amsthm}
\usepackage{amsmath}
\usepackage{graphicx}
\usepackage{marvosym}
\usepackage{empheq}
\usepackage{latexsym}
\usepackage{fontenc}
\usepackage{color}

\renewcommand\le{\leqslant}
\renewcommand\ge{\geqslant}
\renewcommand\to{\rightarrow}

	\def\<{\langle }
	\def\>{\rangle }

\usepackage{amssymb}
\newtheorem{theor}{Theorem}[section]
\newtheorem{prop}[theor]{Proposition}
\newtheorem{lemma}[theor]{Lemma}
\newtheorem{coro}[theor]{Corollary}
\newtheorem{conj}[theor]{Conjecture}
\newtheorem{problem}[theor]{Problem}
\theoremstyle{definition}
\newtheorem{defi}[theor]{Definition}
\theoremstyle{remark}
\newtheorem{rema}[theor]{Remark}

\newtheorem{nota}[theor]{Notation}

\newtheorem*{claim*}{Claim}
\newtheorem*{qu*}{Question}

\newcommand{\genA}{\langle A\rangle}
\newcommand{\aab}{m}
\newcommand{\Zdd}{\mathbb Z^2}
\newcommand{\Zdt}{\mathbb Z^3}
\newcommand{\Zd}{\mathbb Z^d}

\newcommand{\pp}{\mathbb P}
\newcommand{\p}{\mathbb P}

\newcommand{\U}{\mathcal U}
\newcommand{\UU}{\mathcal U}
\newcommand{\dhp}{\mathbb H_u}
\newcommand{\Ss}{\mathcal{S}}

\usepackage{dsfont}
\usepackage{caption}
\usepackage{subcaption}
\usepackage{pifont}
\begin{document}

\title{Anisotropic bootstrap percolation in three dimensions}
\author{Daniel Blanquicett}

\address{Mathematics Department,
University of California, Davis, CA 95616, USA}
\email{drbt@math.ucdavis.edu}
\thanks{{\it Date}: August 30, 2019.
\indent The author was partially supported by CAPES, Brasil.}
	
\begin{abstract}
Consider a $p$-random subset $A$ of initially infected vertices in the discrete cube $[L]^3$,
and assume that the neighbourhood of each vertex consists of the $a_i$ nearest neighbours
in the $\pm e_i$-directions for each
$i \in \{1,2,3\}$, where $a_1\le a_2\le a_3$.
Suppose we infect any healthy vertex $x\in [L]^3$ already having $a_3+1$ infected neighbours,
and that infected sites remain infected forever.
In this paper we determine the critical length for percolation up to a constant factor in the exponent, 
for all triples $(a_1,a_2,a_3)$.
To do so, we introduce a new algorithm called the {\it beams process}
and prove an exponential decay property for a family of subcritical two-dimensional bootstrap processes.
\end{abstract}
	
\maketitle 
\section{Introduction}
The study of bootstrap processes on graphs was initiated in 1979 by Chalupa,
Leath and Reich~\cite{ChLR79}, and is motivated by problems arising from statistical physics, such as the Glauber dynamics of 
the zero-temperature Ising model, and kinetically constrained spin models of the liquid-glass transition 
(see, e.g.,~\cite{FSS02,Morris09,MMT18}, and the recent survey~\cite{Morris17}). 
The $r$-neighbour bootstrap process on a locally finite graph $G$ is a monotone cellular automata on the 
configuration space $\{0,1\}^{V(G)}$, (we call vertices in state $1$ ``infected"), evolving in discrete time
in the following way: $0$ becomes $1$ when it has at least $r$ neighbours in state $1$, and infected vertices remain infected forever.
Throughout this paper, $A$ denotes the initially infected set, and we write $\genA=G$
if the state of each vertex is eventually 1.

We will focus on \emph{anisotropic} bootstrap models, which are $d$-dimensional analogues of a family of
(two-dimensional) processes studied by Duminil-Copin, van Enter and Hulshof \cite{EH07,DCE13,DEH18}.
In these models the graph $G$ has
vertex set $[L]^d$, and the neighbourhood of each vertex consists of the $a_i$ nearest neighbours in the
$-e_i$ and $e_i$-directions for each $i \in [d]$,
where $a_1\le \cdots\le a_d$ and $e_i\in\Zd$ denotes the $i$-th canonical unit vector.
In other words, $u,v\in [L]^d$ are neighbours if (see Figure \ref{figanis3d} for $d=3$)
\begin{align}\label{neigh3} 
u-v\in N_{a_1,\dots,a_d}:=\{\pm e_1,\dots, \pm a_1e_1\}\cup \cdots \cup \{\pm e_d,\dots, \pm a_de_d\}.
\end{align}
We also call this process the $\mathcal N_r^{a_1,\dots, a_d}$-{\it model}.
Our initially infected set $A$ 
is chosen according to the Bernoulli product measure $\p_p=\bigotimes_{v\in [L]^d}$Ber$(p)$,
and we are interested in the so-called {\it critical length for percolation},
for small values of $p$
\begin{equation}\label{criticalL}
 L_c(\mathcal N_r^{a_1,\dots,a_d},p):= \min\{L\in\mathbbm N: \pp_p(\genA=G
 )\ge 1/2\}.
 \end{equation}

 
The analysis of these bootstrap processes for $a_1=\cdots= a_d=1$ was initiated by Aizenman and Lebowitz~\cite{AL88} in 1988,
who determined the magnitude of the critical length 
up to a constant factor in the exponent for the $\mathcal N_2^{1,\dots,1}$-model (in other words, they determined the
`metastability threshold' for percolation). In the case $d = 2$, Holroyd~\cite{H03} determined (asymptotically, as $p \to 0$) the constant in the exponent  (this is usually called a sharp metastability threshold), proving that
$$L_c(\mathcal N_2^{1,1},p) = \exp\bigg( \frac{\pi^2/18 + o(1)}{p} \bigg).$$

For the general $\mathcal N_r^{1,\dots,1}$-model with $2\le r\le d$, the threshold was determined by Cerf and Cirillo \cite{CC99} and Cerf
and Manzo \cite{CM02}, and the sharp threshold by Balogh, Bollob\'as and Morris \cite{BBM09}
and Balogh, Bollob\'as, Duminil-Copin and Morris  \cite{BBDM12}: for all $d\ge r\ge 2$ there exists a computable constant 
$\lambda(d,r)$ such that, as $p\to 0$,
\begin{equation*}
 L_c(\mathcal N_r^{1,\dots,1},p) = \exp_{(r-1)}\bigg(\frac{\lambda(d,r) + o(1)}{p^{1/(d-r+1)}}\bigg).
\end{equation*}

In dimension $d=2$, we write $a_1=a, a_2=b$, and the $\mathcal N_r^{a,b}$-model is called isotropic when $a=b$ and anisotropic when $a<b$.
Hulshof and van Enter \cite{EH07} determined the threshold for the first interesting anisotropic model given by the family $\mathcal N_{3}^{1,2}$, and the 
corresponding sharp threshold was determined by Duminil-Copin and van Enter  \cite{DCE13}: for $b\ge 2$, as $p\to 0$,
\begin{equation*}
L_c\left(\mathcal N_{b+1}^{1,b},p\right)= \exp\left(\left(\frac{(b-1)^2}{4(b+1)} + o(1)\right)\frac{(\log p)^2}{p}\right).
\end{equation*}

The threshold was also determined in the general case $r=a+b$ by van Enter and Fey 
 \cite{AA12} and the proof can be extended to all $b+1\le r\le a+b$: as $p\to 0$,
\begin{equation}\label{paso1}
\log L_c\left(\mathcal N_{r}^{a,b},p\right)=
\begin{cases}
\Theta\left(p^{-(r-b)}\right) & \textup{if }b=a,\\
\Theta\left(p^{-(r-b)}(\log p)^2\right) & \textup{if }b>a.
\end{cases}
\end{equation}
 
\subsection{Anisotropic bootstrap percolation on $[L]^3$}
In this paper we consider the three-dimensional analogue of the anisotropic bootstrap process studied by Duminil-Copin,
van Enter and Hulshof. 
In dimension $d=3$, we write $a_1=a, a_2=b$ and $a_3=c$.
\vskip -.2cm
\begin{figure}[ht]
	\centering
	\includegraphics[width=0.35\textwidth]{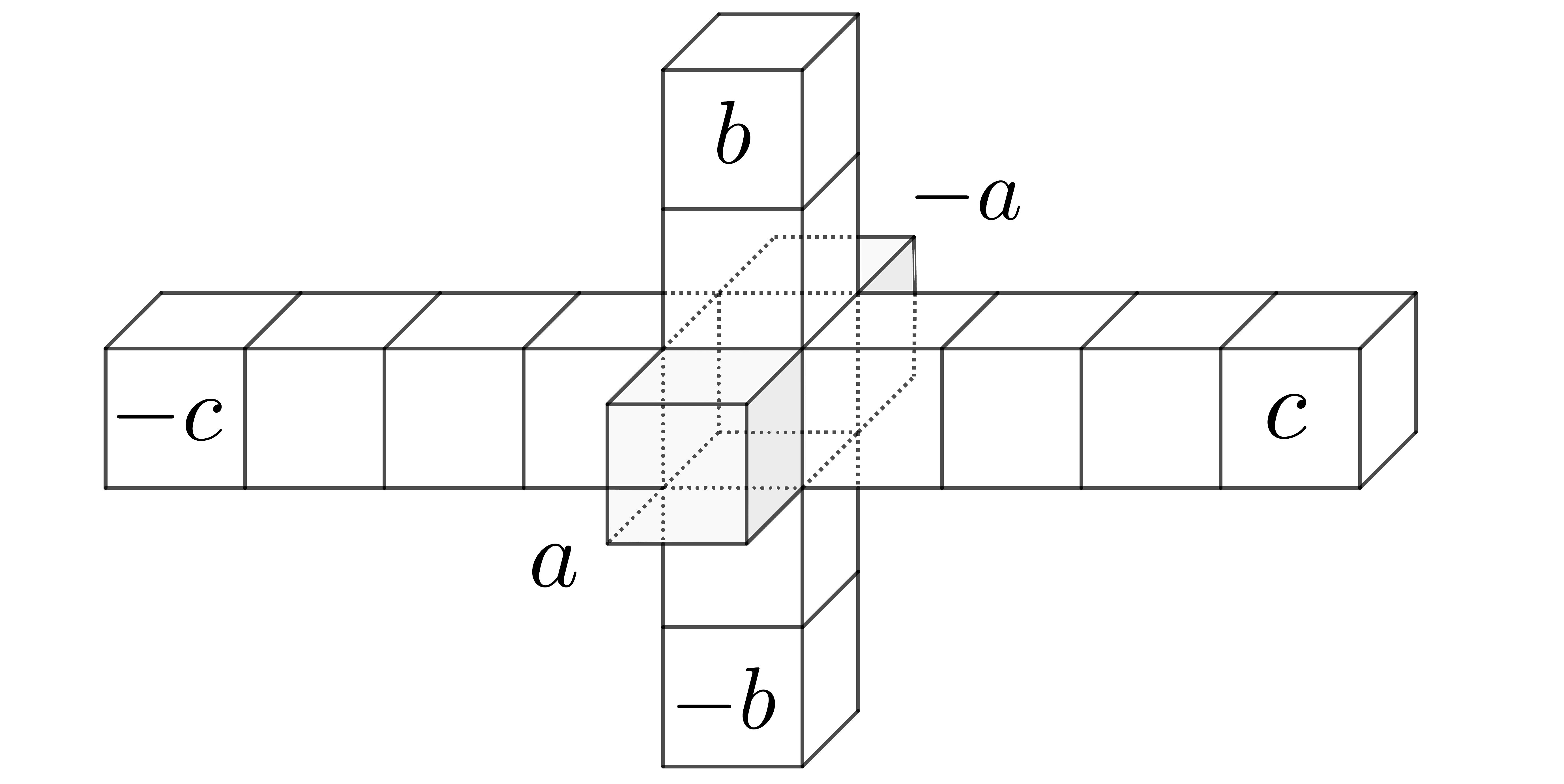}
	\caption{The neighbourhood $N_{a,b,c}$ with $a=1, b=2$ and $c=4$. The $e_1$-axis is towards the reader, the 
	$e_2$-axis is vertical, and the $e_3$-axis is horizontal.}
	\label{figanis3d}
\end{figure}
These models were studied by van Enter and Fey \cite{AA12} for $r=a+b+c$; 
they determined the following bounds on the critical length, 
as $p\to 0$,

\begin{equation}\label{vefey}
		\log \log L_c\left(\mathcal N_{a+b+c}^{a,b,c},p\right)= 
	\begin{cases}
	\Theta\left(p^{-a}\right)                    & \textup{if } b=a, \\
	\Theta\left(p^{-a}(\log \frac 1p)^{2}\right) & \textup{if } b>a.
	\end{cases}
	\end{equation}

Note that, by (\ref{vefey}) the critical length is doubly exponential in $p$ when $r=a+b+c$. 
It is not difficult to show that the critical length is polynomial in $p$ if $r\le c$. 

On the other hand, the critical length is singly exponential in the case $r=c+1$; indeed,
we determine the magnitude of the critical length up to a constant factor in the exponent. 

The following is our main result.
\begin{theor}\label{target0}
As $p\to 0$, 
\begin{equation}\label{magnitud}
\log L_c\left(\mathcal N_{c+1}^{a,b,c},p\right)=  \begin{cases}
\Theta\left(p^{-1/2}\right)  & \textup{if } c=b=a, \\
\Theta\left(p^{-1/2}(\log \frac 1p)^{1/2}\right) & \textup{if } c=b>a, \\
\Theta\left(p^{-1/2}(\log \frac 1p)^{3/2}\right) & \textup{if } c\in \{b+1,\dots,a+b-1\}, \\
\Theta\left(p^{-1}\right)                                  &   \textup{if } c=a+b,   \\
\Theta\left(p^{-1}(\log \frac 1p)^{2}\right)   & \textup{if } c> a+b. 
	\end{cases}
\end{equation}
\end{theor}
The first three cases of this theorem ($c<a+b$) are obtained by adapting standard ideas used for two-dimensional models.
However, to deal with the lower bounds in the last two cases ($c\ge a+b$), it is necessary to introduce a new algorithm
which we call the {\it beams process}, and to develop new tools in {\it subcritical} bootstrap percolation
(see Theorem \ref{expdecIntro}).


\subsection{The BSU model}
The model we study here is a special case of the following extremely general class of $d$-dimensional monotone cellular automata, which were introduced by Bollob\'as, Smith and Uzzell~\cite{BSU15}.

Let $\U=\{X_1,\dots,X_m\}$ be an arbitrary finite family of finite
subsets of $\Zd\setminus \{0\}$. We call $\U$ the {\it update family}, 
each $X\in\U$ an {\it update rule}, and the process itself $\U${\it-bootstrap percolation}.
Let $\Lambda$ be either $\Zd$ or $\Zd_L$ (the $d$-dimensional torus of sidelength $L$).
Given a set $A\subset \Lambda$ of initially {\it infected} sites, set $A_0=A$, and define for each $t\ge 0$,
\[A_{t+1}=A_t\cup\{x\in\Lambda: x+X\subset A_t \text{ for some }X\in\U\}.\]
The set of eventually infected sites is the {\it closure} of $A$, denoted by
$\genA_\U=\bigcup_{t\ge 0}A_t$, and
we say that there is {\it percolation} when $\genA_\U=\Lambda$.


 Let $S^{d-1}$ be the unit $(d-1)$-sphere and denote the discrete half space orthogonal to $u\in S^{d-1}$ as
 $\dhp^d:=\{x\in\Zd:\langle x,u\rangle <0\}$.
The {\it stable set} $\Ss=\Ss(\U)$ is the set of all $u\in S^{d-1}$
such that no rule $X\in\U$ is contained in $\dhp^d$. 
Let $\mu$ denote the Lebesgue measure on $S^{d-1}$. The following classification of families was proposed in \cite{BSU15} for $d=2$ and extended to all dimensions in \cite{BDMS15}:
A family $\U$ is

\begin{itemize}
 \item {\it subcritical} if for every hemisphere $\mathcal H \subset S^{d-1}$ we have $\mu(\mathcal H \cap\Ss)>0$.
 \item {\it critical} if there exists a hemisphere $\mathcal H \subset S^{d-1}$ such that $\mu(\mathcal H \cap\Ss)=0$, and
every open hemisphere in $S^{d-1}$ has non-empty intersection with $\Ss$;
 \item {\it supercritical} otherwise. 
 \end{itemize}

Subcritical families exhibit a behavior which resembles models in classical site percolation,
(see e.g. \cite{BBPS16,H19+}).
For a certain class of subcritical models, we have succeeded in proving an exponential decay property
about the cluster size (see Section \ref{SectionexpdecayK}): denote by $\mathcal K$ the connected component containing $0$ in
$\langle A\rangle_\UU$. 
\begin{theor}\label{expdecIntro}
Assume that $d=2$. Consider 
$\UU$-bootstrap percolation with $\Ss(\UU)=S^1$ and $A\sim \bigotimes_{v\in \Zdd}\textup{Ber}(p)$. If $p$ is small enough, then
	\begin{equation}
	\p_p(|\mathcal K|\ge n)\le 
     e^{-\Omega(n)},
	\end{equation}
	for every $n\in\mathbb N$.
\end{theor}

For dimension $d=2$, Bollob\'as, Duminil-Copin, Morris and Smith proved a universality result in \cite{BDMS15}, 
determining the critical length (with $A\sim \bigotimes_{v\in \Zdd_L}$Ber$(p)$)
\[L_c(\U,p):= \min\{L\in\mathbbm N: \pp_p(\genA_\U=\Zd_L)\ge 1/2\},\]
up to a constant factor in the exponent for all two-dimensional critical families $\U$, which we can briefly state as follows.
\begin{theor}[Universality]
	Let $\U$ be a critical two-dimensional family. There exists a computable positive integer
	$\alpha=\alpha(\U)$ such that, as $p\to 0$, either
	\begin{equation}
	\log L_c(\U,p) =\Theta(p^{-\alpha}),
	\end{equation}
	or
	\begin{equation}
	\log L_c(\U,p) =\Theta(p^{-\alpha}(\log \tfrac 1p)^2).
	\end{equation}
\end{theor}
 
 Proving a universality result of this kind for three (or higher) dimensions is a challenging open problem.
 However, there is a weaker conjecture concerning all critical families and all $d\ge 3$, stated by the
 authors in \cite{BDMS15}; here for simplicity we state only the case $d=3$.
\begin{conj} Let $\UU$ be a critical three-dimensional family. 
As $p\to 0$, either
	\begin{equation}\label{2cri}
	 \log L_c(\UU,p)=p^{-\Theta(1)},
	\end{equation}
	or
	\begin{equation}\label{3cri}
	 \log \log L_c(\UU,p)=p^{-\Theta(1)}.
	\end{equation}
\end{conj}
Let us say that $\UU$ is $2$-{\it critical} if it satisfies condition (\ref{2cri}),
and is $3$-{\it critical} if it satisfies condition (\ref{3cri}).
Observe that we can also think of our $\mathcal N_{r}^{a,b,c}$-model as $\mathcal N_{r}^{a,b,c}$-bootstrap percolation,
where $\mathcal N_{r}^{a,b,c}$ is the family consisting of all subsets of size $r$ of the neighbourhood 
$N_{a,b,c}$ in (\ref{neigh3}).
It is easy to check that the family $\mathcal N_{r}^{a,b,c}$ is critical if and only if 
 \[r\in\{c+1,\dots,a+b+c\}.\]
Moreover, it turns out that the family $\mathcal N_{r}^{a,b,c}$ is 2-critical for all $r\in\{c+1,\dots,c+b\}$
(see Remark \ref{coveredcases}).
On the other hand, the family $\mathcal N_{a+b+c}^{a,b,c}$ is 3-critical by (\ref{vefey});
it is natural to conjecture that this is the case for all $r\in \{c+b+1,\dots,c+b+a\}$.

\subsection{Outline of the proof}
The proofs of all upper bounds are obtained by adapting standard arguments in bootstrap percolation;
the same is true for the lower bounds in the cases $c<a+b$. 

We deal with the lower bounds in the cases $c\ge a+b$ by introducing an algorithm
that we call the {\it beams process}, which will allow us to control the size of the components that can be created in the intermediate steps of the bootstrap dynamics, the trick will be to cover such components with beams (a {\it beam} is a finite 3-dimensional set of the form $H\times[w]$, where $H\subset\Zdd$ is 
 connected and $\langle H\rangle_{\mathcal N_{a+b+1}^{a,b}}=H$, see Definition \ref{beam}).
All initially infected sites are beams, and at every step we merge beams that are within some constant distance, to create a bigger one, then repeat this algorithm and stop it at some finite time; each beam created during the process we call {\it covered}. When we observe the induced process along the $e_3$-direction, it looks like subcritical two-dimensional $\mathcal N_{a+b+1}^{a,b}$-bootstrap percolation,
thus, we can couple the original process and apply the exponential decay property (Theorem \ref{expdecIntro})
 to bound the probability of a beam been covered.

Theorem \ref{expdecIntro} provides new machinery in subcritical bootstrap percolation, we prove it in Section \ref{SectionexpdecayK},
and here we summarize the core idea. First, we need to guarantee the existence of {\it inwards stable droplets},
which are, basically, discrete polygons that can not be infected from outside;
it is possible to show the existence of such droplets by considering families $\U$ such that $\Ss(\U)=S^1$. After that, we combine ideas used by Bollob\'as-Riordan 
in classical percolation models to prove that, when the density of initially infected sites is small enough, then the size of the cluster containing the origin decays exponentially fast, in distribution.

\section{Upper bounds}\label{SectionUpper12}
To prove upper bounds, it is enough to give one possible way of growing from $A$
step by step until we fill the whole of $[L]^3$. 
The case $c>a+b$ will be deduced in the Appendix as a particular case of Proposition
\ref{genuppbound} (see Remark \ref{coveredcases}). On the other hand, the proof of case $c=b=a$
is similar to the proof given in \cite{AL88} and we will omit here.
Hence, we will focus only on the remaining upper bounds in Theorem \ref{target0} in increasing order of technicality.

More precisely, we will give a full proof of the case
$c\in\{b+1,\dots, a+b-1\}$ in Subsection \ref{fullproof}, then we will only sketch the cases
$c=a+b$ and $c=b>a$ in Subsections \ref{sketch1} and \ref{sketch2}, respectively, by pointing out the
small differences between these cases.
\begin{defi}\label{intfilled}
A {\it rectangle} is a set of the form $R=[x]\times[y]\times[w]\subset\Zdt$. 
We say that a rectangle $R$ is {\it internally filled} if $R\subset \langle A\cap R\rangle_{\mathcal N_{r}^{a,b,c}}$,
and denote this event by $I^\bullet(R)$.
\end{defi}

\subsection{Case $c\in\{b+1,\dots, a+b-1\}$}\label{fullproof}
In this section we consider the families
$\mathcal N_{c+1}^{a,b,c},$
with $c\in\{b+1,\dots, a+b-1\}$ (here $a>1$, otherwise this case does not exist).
As usual in bootstrap percolation, we actually prove a stronger proposition.
\begin{prop}\label{upper2}
Fix $c\in\{b+1,\dots, a+b-1\}$ and  consider $\mathcal N_{c+1}^{a,b,c}$-bootstrap percolation.
There exists a constant $\Gamma=\Gamma(c)>0$ such that, if \[L=\exp\left(\Gamma p^{-1/2}(\log \tfrac 1p)^{3/2}\right),\] then
 $\p_p\left(I^\bullet([L]^3)\right)\to 1,\ as\ p\to 0.$
\end{prop} 
When $h,w\ge c$, for simplicity we denote the event
\[I(h,w) := I^\bullet([h]^2\times[w]).\]

\begin{lemma}\label{lemaesencial}
If $p$ is small enough, then 
\[\p_p(I(h,w+1)|I(h,w))\ge 1-e^{-p h^2},\]
under $\mathcal N_{c+1}^{a,b,c}$-bootstrap percolation.
\end{lemma}
\begin{proof}
If $R_1:=[h]^2\times[w]$ is completely infected, we just need to infect the right-most face
$Q:=[h]^2\times\{w+1\}$, and since we have $c$ already infected vertices in $R_1$, then it is enough to find
$1$ infected vertex in $Q$ (see Figure \ref{fig3ladosuper} below). Thus,
\begin{align*}
\pp_p\left(I^\bullet\left([h]^2\times[w+1]\right)\Big|
I^\bullet\left(R_1\right)\right) 
  \ge 
1-\prod_{v\in Q}\left(1-\pp_p\left(v\in A\right)\right)
\ge 1-e^{-p h^2}.
\end{align*}
\end{proof}
Lemma \ref{lemaesencial} tells us the cost of growing one step along
the (easiest) $e_3$-direction, and we are also interested in computing the cost of 
growing along the $e_1$ and $e_2$ (harder) directions. To do so, we will consider general values of $r$: let us first consider the regime
$r\le a+b,$
this implies that given any rectangle $R$, all three induced 2-dimensional processes in the faces of $R$, namely,
$\mathcal N_{r-c}^{a,b}$, $\mathcal N_{r-b}^{a,c}$ and $\mathcal N_{r-a}^{b,c}$, are supercritical.
\begin{lemma}[Supercritical faces]\label{regimer<a+b}
If $r\le a+b$, and $p$ is small enough, then
\[\p_p(I(h+1,w)|I(h,w)) \ge
	\left(1-e^{-{c\choose 2}^{-1}p^{{c\choose 2}}wh}\right)^{2},\]
under $\mathcal N_{r}^{a,b,c}$-bootstrap percolation.
\end{lemma}
\begin{proof}
For $s=a,b$, let $\Delta_s$ be the discrete right-angled triangle 
whose legs are $[r-s]\times\{1\}$ and $\{1\}\times [r-s]$.
Once $R_1=[h]^2\times[w]$ is completely full, to get $R_2=[h+1]^2\times[w]$ internally filled it is enough
to have one copy of $\Delta_a$ in $A\cap(\{h+1\}\times[h]\times[w])$ (front face), and one copy of $\Delta_b$
in $A\cap([h]\times\{h+1\}\times[w])$ (top face, see Figure \ref{fig3ladosuper}).

\begin{figure}[ht]
\centering
\begin{subfigure}{.5\textwidth}
  \centering
  \includegraphics[width=1.05\linewidth]{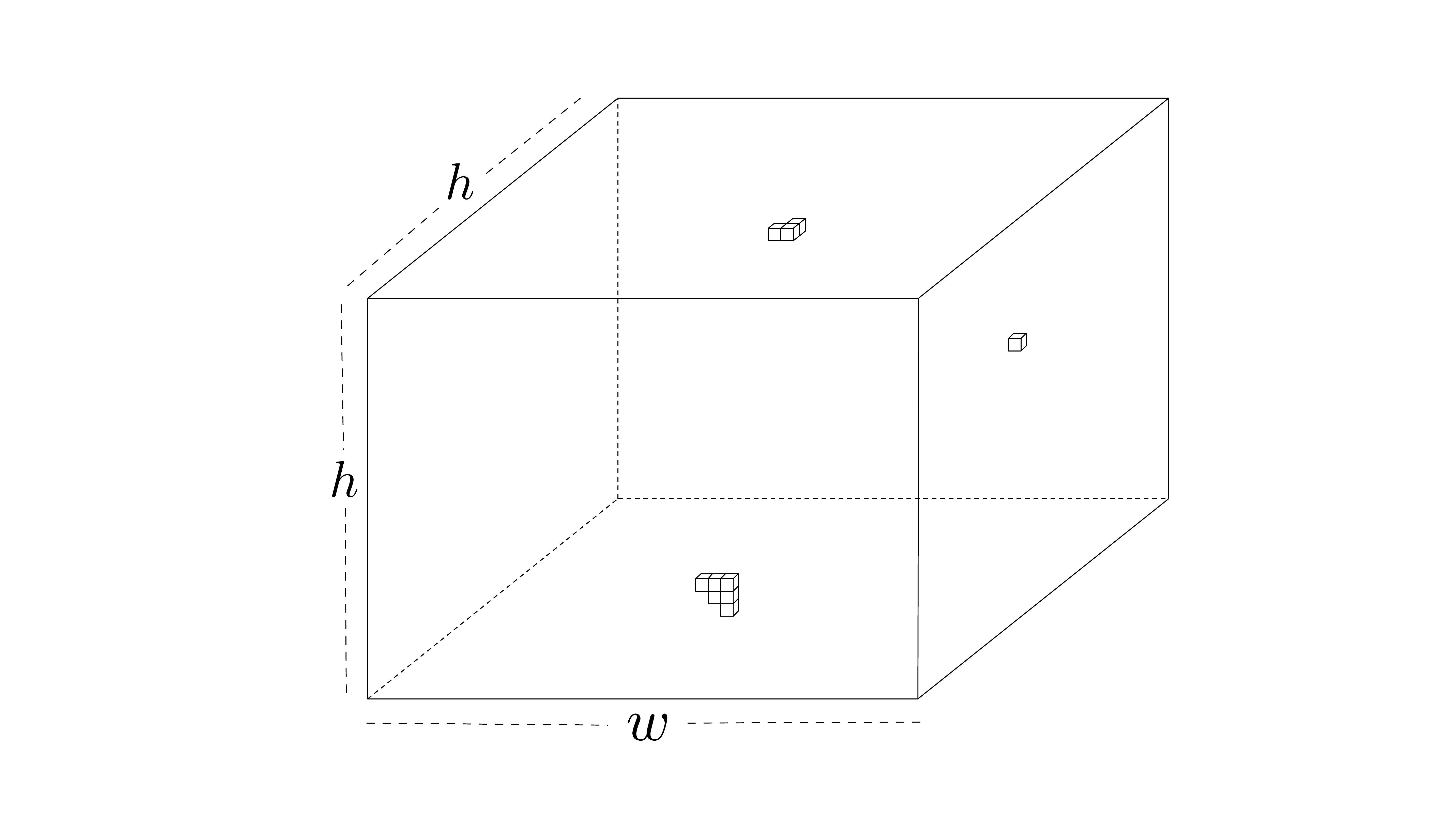}
  \caption{$r\le a+b$}
  \label{fig3ladosuper}
\end{subfigure}%
\begin{subfigure}{.5\textwidth}
  \centering
  \includegraphics[width=1.05\linewidth]{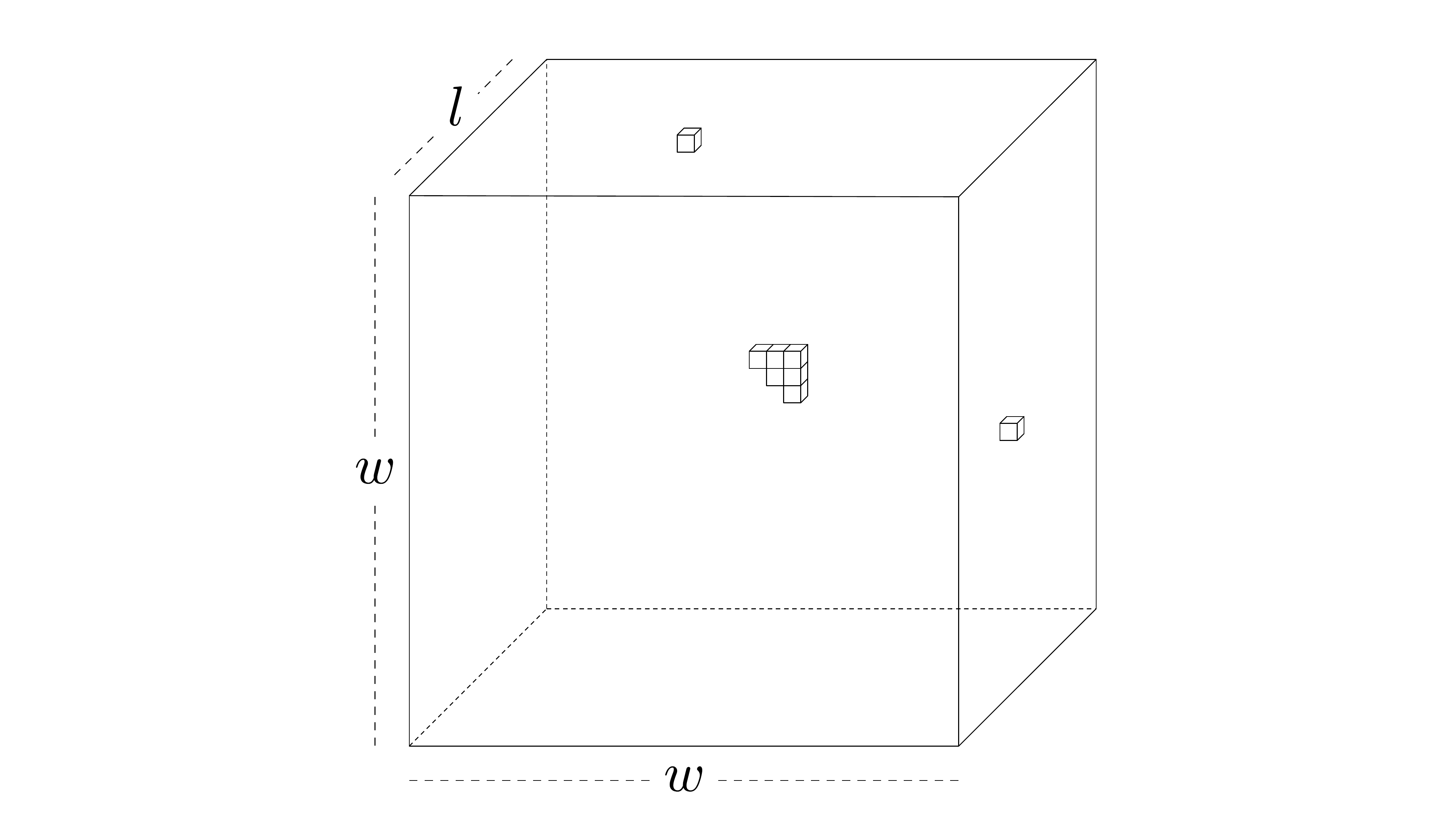}
  \caption{$c=b>a$}
 \label{fig3ladosuperc=b}
\end{subfigure}
\caption{A single vertex on the right-most side, one copy of $\Delta_a$ on the front side, and one copy of $\Delta_b$ on the top side.}
\end{figure}

Since $|\Delta_s|=(r-s)(r-s+1)/2$ and $a\ge 2$, then
$|\Delta_b|\le |\Delta_a|\le {c\choose 2}$.   Hence, by independence between the front and top faces,
\begin{align*}
\pp_p(I^\bullet(R_2)|I^\bullet(R_1)) 
& \ge \left(1-e^{-|\Delta_a|^{-1}p^{|\Delta_a|}wh}\right)
\left(1-e^{-|\Delta_b|^{-1}p^{|\Delta_b|}wh}\right)
  \ge \left(1-e^{-{c\choose 2}^{-1}p^{{c\choose 2}}wh}\right)^{2}.
\end{align*}
\end{proof}

The next step is to determine the size of a rectangle (usually called {\it critical droplet}) such that, once it is internally filled, then it can grow until $[L]^3$ with high probability.
\begin{lemma}\label{secondstep}
Let $\Gamma>0$ be a large constant and set 
$h=cp^{-\frac{1}{2}}(\log \frac 1p)^{\frac 12}$, $R_1:=[h]^2\times 
[c]$ and
\[L=\exp\left(\Gamma p^{-\frac{1}{2}}(\log \tfrac 1p)^{\frac 32}\right).\] Conditionally on $I^\bullet(R_1)$,
the probability of $I^\bullet([L]^3)$ goes to $1$, as $p\to 0$.
\end{lemma}
\begin{proof}
	Consider the rectangles $R_2\subset R_3\subset R_4\subset R_5:= [L]^3$ containing $R_1$, defined by
	\[R_2:=[h]^2\times [c^2p^{-{c\choose 2}+\frac{1}{2}}(\log \tfrac 1p)^{\frac 12}],\
\ \ R_3:=[h^2]^2\times [c^2p^{-{c\choose 2}+\frac{1}{2}}(\log \tfrac 1p)^{\frac 12}],\
\ \ R_4:=[h^2]^2\times [L].\]
	Note that
$\p_p(I^\bullet([L]^3)|I^\bullet(R_1))\ge\prod_{k=1}^{4} \p_p(I^\bullet(R_{k+1})|I^\bullet(R_k)).$
We apply Lemma \ref{lemaesencial} to deduce
	\begin{align*}
	\p_p(I^\bullet(R_{2})|I^\bullet(R_1)) & \ge \left(1-e^{-ph^2}\right)^{c^2p^{-{c\choose 2}+\frac{1}{2}}(\log \frac 1p)^{\frac 12}} 
 \ge e^{-2p^{\frac{c^{2}}{2}-{c\choose 2}}}\to 1,
	\end{align*}
	and by Lemma \ref{regimer<a+b},
\begin{align*}
\p_p(I^\bullet(R_{3})|I^\bullet(R_2)) & \ge 
\bigg(1-e^{-\Omega\left(p^{{c\choose 2}} p^{-{c\choose 2}+\frac{1}{2}}(\log \frac 1p)^{\frac 12}\cdot h\right)}
\bigg)^{2h^2}
 \ge \exp\left(-4h^2p^{2c}\right)\to 1,
\end{align*}
We apply these lemmas again to get $\p_p(I^\bullet(R_{4})|I^\bullet(R_3)) \to 1,$
since $ph^4\ge p^{-1}\gg \Gamma p^{-\frac{1}{2}}(\log \frac 1p)^{\frac 32}$, and also
$\p_p(I^\bullet(R_{5})|I^\bullet(R_4))\to 1$.
We conclude that $\p_p(I^\bullet([L]^3)|I^\bullet(R_1))\to 1$, as $p\to 0$.
\end{proof}
Now, we are ready to show the upper bound.
\begin{proof}[Proof of Proposition \ref{upper2}]
	Set $L=\exp\left(\Gamma p^{-\frac{1}{2}}(\log \frac 1p)^{\frac 32}\right)$, where $\Gamma>0$ is a large constant to be chosen.
	Consider the rectangle
	\[R:=\left[cp^{-\frac{1}{2}}(\log \tfrac 1p)^{\frac 12}\right]^2\times [c]\subset[L]^3,\]
 and the events  $F_L:=\{\exists$ an internally filled copy of $R$ in $[L]^3\}$, 
and $G_L:=\{\langle A\cup R\rangle=[L]^3\}$.
It follows that $\p_p\left(I^\bullet([L]^3)\right)\ge \p_p(F_L)\p_p(G_L|I^\bullet(R)),$
and $\p_p(G_L|I^\bullet(R))\to 1$, as $p\to 0$,
by the previous lemma. 
Therefore, it remains to show that $\p_p(F_L)\to 1$ too. 
 
Indeed, we claim that there exists a constant $C'>0$ such that
	\begin{equation}\label{spcd'}
	\p_p(I^\bullet(R))\ge \exp\left(-C' p^{-\frac{1}{2}}(\log \tfrac 1p)^{\frac 32}\right),
	\end{equation}
	so using the fact that there are roughly $L^3/|R|$ disjoint (therefore independent) copies of $R$ 
	(which we label $Q_1,\dots, Q_{L^3/|R|}$),
	and $|R| 
	\le p^{-3}$, (\ref{spcd'}) immediately gives

\begin{align*}
	\p_p(F_L^c)& \le \p_p\left(\bigcap_i   
	I^\bullet(Q_i)^c\right) 
 \le \left[1-\p_p(I^\bullet(R))\right]^{L^3/|R|} 
 \le \exp\left(-e^{3\log L-3\log(1/p)-C'p^{-\frac{1}{2}}(\log \tfrac 1p)^{\frac 32}}\right).
	\end{align*}
Since $\log L= \Gamma p^{-\frac{1}{2}}(\log \tfrac 1p)^{\frac 32}$, by taking $\Gamma>C'/3$ we conclude $\p_p(F_L)\to 1$, as $p\to 0$.
To finish, it is only left to prove inequality (\ref{spcd'}).

In fact, note that a way to make $R$ be internally filled is the following: start with $[c]^3\subset A$, and then grow from $R_k=[k]^2\times[c]$ to $R_{k+1}$, for
$k=c,\dots, m:=cp^{-\frac{1}{2}}(\log \tfrac 1p)^{\frac 12}.$
This gives us
\begin{align*}
  \p_p\left(I^\bullet(R)\right) & \ge
  \p_p([c]^3\subset A)\prod_{k=c}^m\p_p\left(I^\bullet(R_{k+1})|I^\bullet(R_k)\right)
 \ge p^{c^3}\prod_{k=c}^m
  \left(1-e^{-{c\choose 2}^{-1}p^{{c\choose 2}}ck}\right)^{2} \\
&\ge p^{c^3+c^2m}
 \ge \exp\left(-C' p^{-\frac{1}{2}}(\log \tfrac 1p)^{\frac 32}\right),
\end{align*}
for $C'>c^3$, as we claimed.
\end{proof}

\subsection{Case $c=a+b$}\label{sketch1}
In this section we consider the families $\mathcal N_{a+b+1}^{a,b,a+b},$
corresponding to the case  $r=a+b+1$.
To do so, we first compute the cost of growing for all cases $a+b<r\le a+c$,
where, the induced $\mathcal N_{r-c}^{a,b}$ process is still supercritical, but the induced processes $\mathcal N_{r-b}^{a,c}$ and $\mathcal N_{r-a}^{b,c}$ are critical.
\begin{lemma}[Critical faces]\label{regimer>a+b}
If $r\in\{a+b+1,\dots, a+c\}$ and $p$ is small, then
\[\p_p(I(h+1,w)|I(h,w))\ge
	\left(1-e^{-\frac{1}{r-a}p^{r-a}w}\right)^{r}
	\left(1-e^{-\frac{1}{\aab}p^{\aab}w}\right)^{2h},\]
under $\mathcal N_r^{a,b,c}$-bootstrap percolation, with $\aab:=r-(a+b).$
\end{lemma}
\begin{proof}
Once $[h]^2\times[w]$ is completely full, to fill $[h+1]^2\times[w]$
		it is enough to have the occurrence of the events $F_h^{e_1}$ and $F_h^{e_2}$ defined as follows: $F_h^{e_1}$ as 
		(growing along the $e_1$-direction) there exist $r-a$ adjacent vertices
		in $A\cap(\{h+1\}\times\{1\}\times[w])$, $r-(a+1)$ adjacent vertices in 
		$A\cap(\{h+1\}\times\{2\}\times[w])$,$\dots$, $r-(a+b-1)$ adjacent vertices in
		$A\cap(\{h+1\}\times\{b\}\times[w])$, and for each $i=b+1,\dots,h$, there exist
		$\aab=r-(a+b)$ adjacent vertices in $A\cap(\{h+1\}\times\{i\}\times[w])$.
\vspace{-.3cm}
\begin{figure}[ht]\label{sec21}
	\centering
	\includegraphics[width=.5\textwidth]{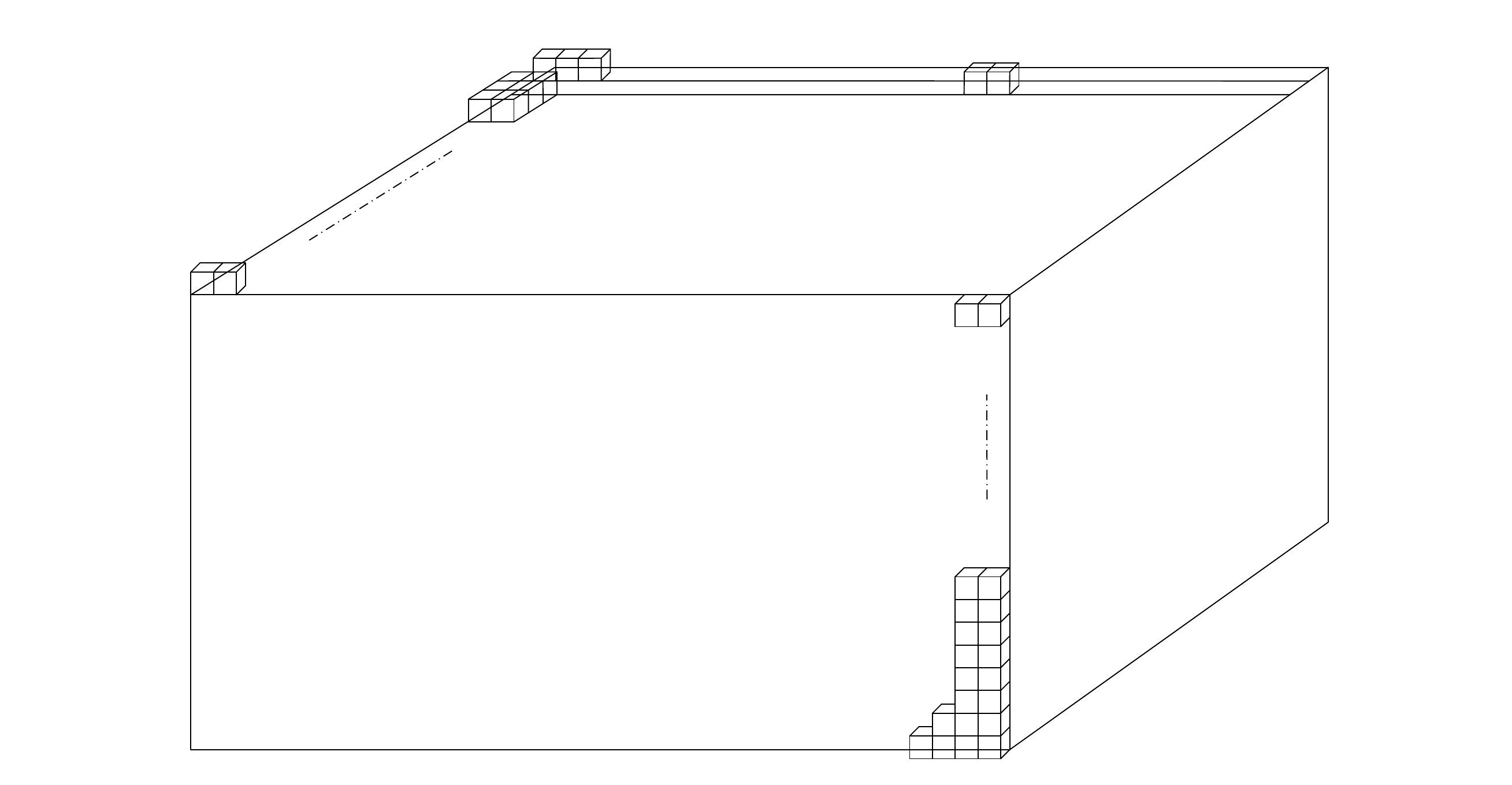}
	\caption{$\aab$ vertices in each of the lines along the $e_3$-direction ($\aab=2$).}
	\label{figcmayora+b}
\end{figure}
\begin{align*}
\pp_p(F_h^{e_1}) &  \ge \prod_{k=\aab+1}^{r-a}\left(1-(1-p^{k})^{\frac wk}\right)
\prod_{i=1}^{h}\left(1-(1-p^{\aab})^{\frac wm}\right)  
 \ge \left(1-e^{-\frac{1}{r-a}p^{r-a}w}\right)^{b}
		\left(1-e^{-\frac{1}{\aab}p^{\aab}w}\right)^{h}.  
\end{align*}
$F_h'$ is defined analogously, this time growing along the $e_2$-direction (see Figure \ref{figcmayora+b}), thus
\begin{align*}
\pp_p(F_h^{e_2}) \ge \left(1-e^{-\frac{1}{r-b}p^{r-b}w}\right)^{a}\left(1-e^{-\frac{1}{\aab}p^{\aab}w}\right)^{h}. 
\end{align*}
Finally, 
$p_r(h+,w)\ge \pp_p(F_h^{e_1})\pp_p(F_h^{e_2})$.
\end{proof}
\begin{rema}
In the regime $a+c<r\le a+b+c$,
all three the induced 2-dimensional processes $\mathcal N_{r-c}^{a,b}$, $\mathcal N_{r-b}^{a,c}$ and $\mathcal N_{r-a}^{b,c}$ are critical.
\end{rema}

As before, we need to set the size of a critical droplet.
\begin{lemma}
Fix $\varepsilon>0$ and let $\Gamma$ be a large constant. Set  $h=p^{-1/2-\varepsilon}$, $R_1:=[h]^2\times [h^{2}]$ and \[L=\exp(\Gamma p^{-1}).\] 
Conditionally on $I^\bullet(R_1)$,
the probability of $I^\bullet([L]^3)$ goes to $1$, as $p\to 0$.
\end{lemma}
The proof of this lemma is very similar to that of Lemma \ref{secondstep}, thus, we omit it.
Finally, to deduce the upper bound, we proceed in the same way that we used to prove Proposition \ref{upper2},
this time by showing that the critical droplet $R_1$ satisfies
\begin{equation}\label{spcd}
	\p_p(I^\bullet(R_1))\ge \exp(-C'p^{-1}),
\end{equation}
for some constant $C'>0$, depending on the integral of the function 
$f_1:(0,\infty)\to(0,\infty)$, defined by $f_1(z)=-\log(1-e^{-z})$ (see e.g. \cite{AL88} and \cite{H03}).


\subsection{Case $c=b>a$}\label{sketch2}
In this section we sketch the proof of the last case. 
Consider the families \[\mathcal N_{c+1}^{a,c,c}.\]

We follow the same steps, taking into account that the way to grow is slightly different:
in this case, to grow along the $e_2$-direction is as easy as 
grow along the $e_3$-direction, so that 
it is enough to find a single infected vertex on the right-most and top sides,
while to grow along the $e_1$-direction we still need to find one copy of $\Delta_a$ on the front
side (see Figure \ref{fig3ladosuperc=b} above).

\begin{lemma}\label{grow2}
	Fix $l,w\ge c$ and let $I=I^\bullet([l]\times[w]^2)$. If $p$ is small enough, then 
	\begin{enumerate}
	\item[(i)] $\p_p\left(I^\bullet([l]\times[w+1]^2)|I\right)\ge \left(1-e^{-plw}\right)^2$.
	\item[(ii)]  $\p_p\left(I^\bullet([l+1]\times[w]^2)|I\right)\ge
	1-e^{-\Omega(p^{c(c+1)/2}w^2)}$. 
	\end{enumerate}
\end{lemma}
\begin{proof}
Similar to the proof of Lemmas \ref{lemaesencial} and \ref{regimer<a+b}.
\end{proof}
The size of the critical droplet is given by the following lemma, again, we omit the proof.
\begin{lemma}
Let $\Gamma$ be a large constant. Set  
 $R_1:=[p^{-1/2} (\log\frac{1}{p})^{-\frac 12}]\times \left[2\Gamma p^{-1}\log\tfrac 1p\right]^2$ and
 \[L=\exp\left(\Gamma p^{-1/2}\sqrt{\log\tfrac{1}{p}}\right).\]
Conditionally on $I^\bullet(R_1)$,
the probability of $I^\bullet([L]^3)$ goes to $1$, as $p\to 0$.
\end{lemma}
Finally, to deduce the upper bound, we proceed as before, this time by showing that 
\begin{equation}\label{spcd}
	\p_p(I^\bullet(R_1))\ge \exp\left(-C'p^{-\frac{1}{2}}\sqrt{\log \tfrac{1}{p}}\right),
\end{equation}
for some constant $C'>0$, depending on $c$ and the function $f_2(z)=-\log(1-e^{-z^2})$.

\section{Lower bounds via components process}\label{SectionLowerComp}  
In this section we only prove the lower bounds corresponding to the cases $c< a+b$, since the proof is an application of the {\it components process} (see Definition \ref{compro} below),
a variant of an algorithm introduced Bollob\'as, Duminil-Copin, Morris, and Smith \cite{BDMS15}. 
The lower bound in the case $a=b=c=1$ was proved in \cite{AL88}, and the general case $a=b=c$ follows by using the same arguments.
Thus, we will omit this case, and prove the following.
\begin{prop}\label{lower1weaker}
If $c>a$, there is a constant $\gamma=\gamma(c)>0$
such that, for 
\[L<\exp\left(\gamma p^{-1/2}(\log\tfrac{1}{p})^{1/2}\right),\]
 $\p_p(I^{\bullet}([L]^3))\to 0,$ as $p\to 0$,
under $\mathcal N_{c+1}^{a,c,c}$-bootstrap percolation.
\end{prop}
\begin{prop}\label{lower2weaker}
If $c\in\{b+1,\dots,a+b-1\}$, there exists $\gamma=\gamma(c)>0$
such that, for 
\[L<\exp\left(\gamma p^{-1/2}(\log\tfrac{1}{p})^{3/2}\right),\]
 $\p_p(I^{\bullet}([L]^3))\to 0,$ as $p\to 0$,
under $\mathcal N_{c+1}^{a,b,c}$-bootstrap percolation.
\end{prop}


\begin{nota}
Throughout this paper, when $\UU=\mathcal N_{r}^{a,b,c}$ we will omit the subscript in the closure and
simply write $\langle \cdot \rangle$ instead of $\langle \cdot \rangle_{\mathcal N_{r}^{a,b,c}}$.
\end{nota}
Aizenman and Lebowitz \cite{AL88} obtained the matching lower bound for the family $\mathcal N_2^{1,1,1}$ by using the so-called {\it rectangles process}, and they exploited the fact that for this model, the closure $\genA$ is a union of rectangles which are separated by distance at least 2.

In our case, the closure $\genA$ is more complicated. 
Thus, we need to introduce a notion about rectangles
which is an approximation to being internally filled, and
this notion requires a strong concept of connectedness; we define both concepts in the following.
Consider the superset of $N_{a,b,c}$ (see  (\ref{neigh3})) given by
 \[\bar N_{a,b,c}:= \{(u_1,u_2,u_3)\in\Zdt: |u_1|\le a, |u_2|\le b, |u_3|\le c
 \textup{ and } u_1u_2u_3=0\}.\]
\begin{defi}\label{strconn}
Let $G=(V,E)$ be the graph with vertex set $[L]^3$ and edge set given by
$E=\{(u,v): u-v\in  \bar N_{a,b,c}\}$.   
 We say that a set $S\subset [L]^3$ is {\it strongly connected}
 if it is connected in the graph $G$.
\end{defi}

\begin{defi}\label{intspa}
We say that the rectangle $R\subset [L]^3$ is 
{\it internally spanned} by $A$, if there exists a strongly connected set $S\subset\langle A\cap R\rangle$
such that $R$ is the smallest rectangle containing $S$. We denote this event by $I^{\times}(R)$.
\end{defi}
Note that when a rectangle is internally filled then it is also internally spanned, therefore,
Propositions \ref{lower1weaker} and \ref{lower2weaker} are consequences of the following results.
\begin{prop}\label{lower1}
If $c>a$, there is a constant $\gamma=\gamma(c)>0$
such that, for 
\[L<\exp\left(\gamma p^{-1/2}(\log\tfrac{1}{p})^{1/2}\right),\]
 $\p_p(I^{\times}([L]^3))\to 0,$ as $p\to 0$,
under $\mathcal N_{c+1}^{a,c,c}$-bootstrap percolation.
\end{prop}

\begin{prop}\label{lower2}
If $c\in\{b+1,\dots,a+b-1\}$, there exists $\gamma=\gamma(c)>0$
such that, for 
\[L<\exp\left(\gamma p^{-1/2}(\log\tfrac{1}{p})^{3/2}\right),\]
 $\p_p(I^{\times}([L]^3))\to 0,$ as $p\to 0$,
under $\mathcal N_{c+1}^{a,b,c}$-bootstrap percolation.
\end{prop}
We will prove them in Sections \ref{subseclower1} and \ref{subseclower2}, respectively.

\subsection{The components process}
The following is an adaptation of the spanning algorithm in \cite[Section 6.2]{BDMS15}.
We will use it to show
an Aizenman-Lebowitz-type lemma, 
which says that when a rectangle
is internally spanned, then it contains internally spanned rectangles of all
intermediate sizes 
(see Lemmas \ref{ALlema0} and \ref{ALlema00} below).
\begin{defi}[The components process]\label{compro}
 Let $A=\{v_1,\dots,v_{|A|}\}\subset [L]^3$ and fix $r\ge c+1$.
 Set $\mathcal R:=\{S_1,\dots,S_{|A|}\}$, where $S_i=\{v_i\}$ for each $i=1,\dots,|A|$.
 Then repeat the following steps until STOP:
 \begin{enumerate}
  \item If there exist distinct sets $S_{1},S_{2}\in\mathcal R$ such that
  \[ S_{1}\cup S_{2} \]
  is strongly connected, then remove them from $\mathcal R$,
  and replace by $\langle S_{1}\cup S_{2} \rangle$.
  \item If there do not exist such sets in $\mathcal R$, then STOP.
 \end{enumerate}
\end{defi}
\begin{rema}\label{stopfinitetime}
We highlight two properties that are due to the way the algorithm evolves:
\begin{itemize}
 \item At any stage of the component process, any set 
$S=\langle S_{1}\cup S_{2} \rangle$ added to the collection
$\mathcal R$ satisfies $S=\langle A\cap S\rangle=\langle S\rangle\subset [L]^3$ (since $r\ge c+1$).
In particular, the smallest rectangle containing $S$ is internally spanned.
 \item Since $G$ is finite, the process stops in finite time; so that we can consider the final
 collection $\mathcal R'$ and set
 $V(\mathcal R')=\bigcup\limits_{S\in\mathcal R'}S$.
\end{itemize}
\end{rema}
\begin{lemma}\label{final=gen}
  $V(\mathcal R')=\genA$.
\end{lemma}
\begin{proof}
 Clearly $A\subset V(\mathcal R')\subset \genA$, and to prove that $\genA\subset V(\mathcal R')$
 we argue by contradiction.
Suppose this is not the case, since $A\subset V(\mathcal R')$, there would exist vertices
 $v\in\genA \setminus V(\mathcal R')$ and $v_1,\dots,v_r\in V(\mathcal R')$ such that $v-v_i\in N_{a,b,c}$, for $i=1,\dots,r$.
Let us say that $v_i\in S_i'$ for some sets $S_i'\in\mathcal R'$.  

Since $S_1'=\langle S_1'\rangle$, $v_k\notin S_1'$ for some $k\ne 1$, 
so that $S_k'\ne S_1'$.  In particular $S_1'\cup S_k'$ is strongly connected via $v_1,v,v_k$ and
 $\langle S_1'\cup S_k' \rangle\notin \mathcal R'$; this contradicts the definition of $\mathcal R'$.
\end{proof}
\begin{nota}
 From now on, we allow some abuse of notation, by denoting as $[x]\times[y]\times[z]$
 any translate of the rectangle $R=[x]\times[y]\times[z]$ located at the origin.
\end{nota}


\subsection{Case $c=b>a$}\label{subseclower1}
The following is a variant of the Aizenman-Lebowitz Lemma in \cite{AL88}.
\begin{lemma}\label{ALlema0}
 Consider $\mathcal N_{r}^{a,b,c}$-bootstrap percolation with $r\ge c+1$.
 If $[L]^3$ is internally spanned then,
 for every \ $h,k\le L$ there exists an internally spanned rectangle
 $[x]\times[y]\times[z]$ inside $[L]^3$ satisfying $(y+z)/2 \le 2ck$, and either
 \begin{enumerate}
  \item[(a)] $x\ge h$, or
  \item[(b)] $x< h$ and $(y+z)/2 \ge k$.
 \end{enumerate}
\end{lemma}
\begin{proof}
Let $S$ be the first set that appears in the components process such that,
the smallest rectangle $Q:=[x]\times[y]\times[z]$ containing $S$ satisfies either $x\ge h$ or
$(y+z)/2\ge k$ (such a set exists since $V(\mathcal R')=\genA$ and $[L]^3$ is internally spanned).
Since $Q$ is internally spanned, it only remains to show that 
the semi-perimeter $(y+z)/2$ is at most $2ck$.

In fact, we know that $S=\langle S_{1}\cup  S_{2} \rangle$ for some 
sets $S_{t}$ such that, for each $t=1,2$, the smallest rectangle 
$[{x_t}]\times[{y_t}]\times[{z_t}]$ containing $S_{t}$ satisfies
$({y_t}+{z_t})/2\le k-1/2$.
Since $S$ is strongly connected, the new semi-perimeter is
\begin{align*}\label{maxdiam2}
	\frac{{y}+{z}}{2} & \le 2 \max_{t= 1,2}\left\{\frac{{y_t}+{z_t}}{2}\right\}
	+\frac {b+c}{2} 
 \le 2c\left(k-\frac 12\right)+ c
= 2ck.
\end{align*}
\end{proof}
\begin{proof}[Proof of Proposition \ref{lower1}]
Fix a small constant $\delta >0$ and take $L<\exp(\gamma p^{-1/2}(\log\tfrac 1p)^{1/2})$, where $\gamma=\gamma(\delta)>0$
is another small constant to be chosen. Let us show that $\p_p(I^{\times}([L]^3))$ goes to $0$, as $p\to 0$. Set
\[h=\delta p^{-\frac 12}(\log\tfrac 1p)^{-\frac 12},\ \ \ k=p^{-\frac 12}\sqrt{\log\tfrac 1p}.\]

If $[L]^3$ is internally spanned, by Lemma \ref{ALlema0} the following event
occurs: there exists an internally
spanned rectangle $Q=[x]\times[y]\times[z]\subset [L]^3$ satisfying $(y+z)/2 \le 2ck$, and either
$x\ge h$, or $x< h$ and $(y+z)/2 \ge k$.

Suppose first that $x< h$ and $(y+z)/2 \ge k$, thus, either $y$ or
$z$ is at least $k$, by symmetry ($b=c$), we can assume $z\ge k$.
Since $Q$ is internally spanned, every copy of the slab 
$[x]\times[y]\times[2c]$ must contain at least $1$ element
of $A$. Consider only the $z/2c$ disjoint slabs that partition $Q$; since $xy =O (hk)$, if $\delta$ is small, the probability of this event is at most
\begin{align*}
  (O(pxy))^{z/2c} & \le (O(phk))^{k/2c} 
= (O(\delta))^{k/2c} 
 \le e^{-k}.
\end{align*}

On the other hand, if $x\ge h$ we use the fact that $a\le (c+1)-2$, thus, 
since $Q$ is internally spanned,
every copy of the slab $[3a]\times[y]\times[z]$ must contain at least $2$ elements
$u,v\in A$ such that $u-v\in N_{a,b,c}$. Since $x\ge h$, the probability of this event is at most
\begin{align*}
  \left(O(p^{2}yz)\right)^{x/3a} &\le \left(O(p^{2}k^2)\right)^{h/3a} 
\le \left(O(p^{2}p^{-1}\log\tfrac 1p)\right)^{h/3a} 
 \le e^{-\Omega(\delta k)}.  
\end{align*}

Therefore, the probability that $Q$ is internally spanned is at most $e^{-c(\delta)k}$
for some small constant $c(\delta)>0$. Finally, denoting by $\mathcal R_{k}$ the 
collection of rectangles $[x]\times[y]\times[z]\subset[L]^3$ satisfying
$y+z\le 4ck$, it follows by union bound that

\begin{align*}
\p_p(I^{\times}([L]^3))
& \le \sum_{Q\in\mathcal R_{k}}\p_p(I^{\times}(Q))
 \le |\mathcal R_{k}|e^{-c(\delta)k}
 \le L^7 \exp\left(-c(\delta)p^{- 1/2}(\log\tfrac 1p)^{1/2}\right) \to 0,
\end{align*}
as $p\to 0$, for $7\gamma<c(\delta)$, and we are finished.
\end{proof}

\subsection{Case $c\in\{b+1,\dots,a+b-1\}$}\label{subseclower2}
In this case, 
the corresponding analogue of the Aizenman-Lebowitz Lemma is as follows.
\begin{lemma}\label{ALlema00}
Consider $\mathcal N_{r}^{a,b,c}$-bootstrap percolation with $r\ge c+1$.
 If $[L]^3$ is internally spanned then,
 for every \ $h,k\le L$ there exists an internally spanned rectangle
 $[x]\times[y]\times[z]\subset [L]^3$ satisfying $(x+y)/2 \le rh$, and either
 \begin{enumerate}
  \item[(a)] $z\ge k$, or
  \item[(b)] $z< k$ and $(x+y)/2 \ge h$.
\end{enumerate}
\end{lemma}
The proof of this lemma is identical to that of Lemma \ref{ALlema0},
we therefore omit it and proceed to the proof of the lower bound.
\begin{proof}[Proof of Proposition \ref{lower2}]
Take $L< \exp(\gamma p^{-\frac 12}(\log\tfrac 1p)^{\frac 32})$, where $\gamma>0$ is some small constant.
Fix $\delta>0$ and set
\[h=\delta p^{-\frac 12}(\log\tfrac 1p)^{\frac 12},\ \ \ k=p^{-1}.\]
If $[L]^3$ is internally spanned, by Lemma \ref{ALlema00}, there is an internally
spanned rectangle $Q=[x]\times[y]\times[z]$ satisfying
$(x+y)/2 \le rh$, and either $z\ge k$, or
$z< k$ and $(x+y)/2 \ge h$.

In the case that $z\ge k$ we also know that $xy = O(h^2)$.
As before, every copy of the slab 
$S:=[x]\times[y]\times[r]$ intersects $A$.
Thus, by considering the $z/r$ disjoint slabs; if $\delta$ is small, 
the probability of this event is at most

\begin{align*}
\p_p\left(S\cap A\ne \emptyset
\right)^{z/r} 
&\le \left(1-e^{-\Omega(ph^2)}\right)^{k/r}
 =\left(1-p^{\Omega(\delta^2)} \right)^{p^{-1}/r}
 \le e^{-p^{-3/4}}.
\end{align*}

In the case that $z< k$ and $(x+y)/2 \ge h$, we can assume w.l.o.g. that $y\ge h$ and use the fact that 
$b\le c-1=r-2$. This time there is no gap along the $e_2$-direction, so, every copy of the slab
$[x]\times [2r]\times[z]$ must contain at least $2$ elements of $A$ within constant distance. The probability of this event is at most

\begin{align*}
  \left(O(p^{2}xz)\right)^{y/2r} &\le \left(O(p^{2}hk)\right)^{h/2r} 
 \le e^{-\Omega(h\log\frac 1p)}.  
\end{align*}
Therefore, the probability that $Q$ is internally spanned is at most 
$ e^{-c(\delta)p^{-\frac 12}(\log\tfrac 1p)^{\frac 32}},$
for some small constant $c(\delta)>0$. 
Denote by $\mathcal R_{h}'$ the 
collection of rectangles $[x]\times[y]\times[z]\subset[L]^3$ satisfying
$x+y \le 2rh$, it follows by union bound that

\begin{align*}
\p_p(I^{\times}([L]^3))
& \le \sum_{Q\in\mathcal R_{h}'}\p_p(I^{\times}(Q))
 \le |\mathcal R_{h}'|e^{-c(\delta)p^{-\frac 12}(\log\tfrac 1p)^{\frac 32}}
 \to 0,
\end{align*}
as $p\to 0$, if $\gamma>0$ is small.
\end{proof}


\section{Exponential decay for subcritical families}\label{SectionexpdecayK}
In this section, we develop new machinery for 
$\UU$-bootstrap percolation in $\Zdd$ with $\UU$ subcritical.
The first paper studying these families in such generality is \cite{BBPS16}, 
it turns out that these families exhibit a behavior which resembles models in classical site percolation,
for instance, in \cite{BBPS16} it is proved that $p_c(\Zdd,\UU)>0$, for every subcritical family $\UU$,
where
\begin{equation*}\label{pcritico}
p_c(\Zdd,\UU):=\inf\{p:\p_p(\genA_\UU = \Zdd)= 1\}.
\end{equation*}

We will only deal with subcritical families $\UU$ satisfying $p_c(\Zdd,\UU)=1$; the authors of \cite{BBPS16}
proved that this condition is equivalent to $\Ss(\UU)=S^1$.
Our aim is to show that for such families, if we choose the initial infected set $A$ to be
$\varepsilon$-random with $\varepsilon$ small enough, then the size of the cluster in $\genA_{\UU}$ containing the origin decays exponentially fast. More precisely
\begin{defi}
We define the {\it component} (or {\it cluster})
of $0\in\Zdd$ as the connected component containing $0$ in the graph induced by 
$\langle A\rangle_\UU$, and we denote it by $\mathcal K=\mathcal K(\UU,A)$.
If $0\notin \genA_\UU$, then we set $\mathcal K=\emptyset$.
\end{defi}
The following is the main result in this section. It will be essential to prove the remaining lower bounds (cases $c\ge a+b$) in Section \ref{SectionLowerBeams}.
\begin{theor}\label{expdecIntro2}
Consider 
$\UU$-bootstrap percolation with $\Ss(\UU)=S^1$. If $p$ is small enough, then
	\begin{equation*}
	\p_p(|\mathcal K|\ge n)\le 
     e^{-\Omega(n)},
	\end{equation*}
	for every $n\in\mathbb N$.
\end{theor}
In order to prove this theorem, first we need to guarantee the existence of {\it inwards stable droplets},
which are, basically, discrete polygons that can not be infected from outside,
it is possible to do so by using the condition $\Ss(\UU)=S^1$. After that, we introduce the {\it dilation radius}, which is a constant depending on $\UU$, used to obtain an extremal lemma that gives us a quantitative measure of the ratio $|\genA_\UU|/|A|$. Finally, we combine ideas used by Bollob\'as and Riordan 
in classical percolation models to conclude.

\subsection{Inwards stable droplets and the dilation radius}
Given $x,y\in\mathbb R^2$ we denote the usual euclidean distance between $x$ and $y$ by $\|x-y\|$,
and $B_\rho(x)$ is the ball of radius $\rho>0$ centered at $x$:
\begin{equation}
  B_\rho(x):=\{y\in\mathbb R^2: \|x-y\|\le \rho\}.  
\end{equation}
For simplicity, we denote $B_\rho:=B_\rho(0)$.
Imagine for a moment that we have a convex set $D$ in the plane
and suppose it is inscribed in $B_\rho$, 
then we know that any other ball with radius $\rho$ and center outside   
$B_{3\rho}$ is disjoint from $D$.
This simple remark will be important to prove 
Theorem \ref{expdecIntro2}
(see Lemma \ref{extremaldeK}).

\begin{defi}\label{tripling}
	Let us define a {\it rounded droplet} $D$ as the intersection of $\Zdd$ with a bounded convex set in the plane.
	We say that $D\subset\Zdd$ is {\it inwards stable} for $\U$ if
	\begin{equation}
	\langle \Zdd\setminus D\rangle_\UU=\Zdd\setminus D.
	\end{equation}
\end{defi}
We need to guarantee the existence of inwards stable (rounded) droplets, note that they are finite; this is the only point where we use the hypothesis $\Ss(\UU)=S^1$.
\begin{lemma}[Existence, \cite{BBPS16}]\label{existeindrop}
	If $\Ss(\UU)=S^1$ then, there exist an inwards stable droplet $D$ such that
	$0\in D$.
\end{lemma}
The origin $0\in\Zdd$ has no special role here, it is just a reference point to locate the droplet $D$.
Any translate of $D$ is inwards stable as well.

There are several choices for the shape of inwards stable droplets. The following proof is included in \cite{BBPS16}, and
shows that $D$ could be a polygon or not; this fact justifies the {\it rounded} term in the definition.
\begin{proof}[Sketch of proof of Lemma \ref{existeindrop}]
Suppose that $B_\rho$ is initially healthy.
	If $\rho$ is large enough then every rule $X\in\U$ can only infect sites in disjoint circular segments
	`cut off' from $B_\rho$ using chords of length at most
\[\nabla(\UU):=\max_{X\in\UU}\max_{x,y\in X}\|x-y\|,\]
\begin{figure}[ht]
	\includegraphics[width=.3\textwidth]{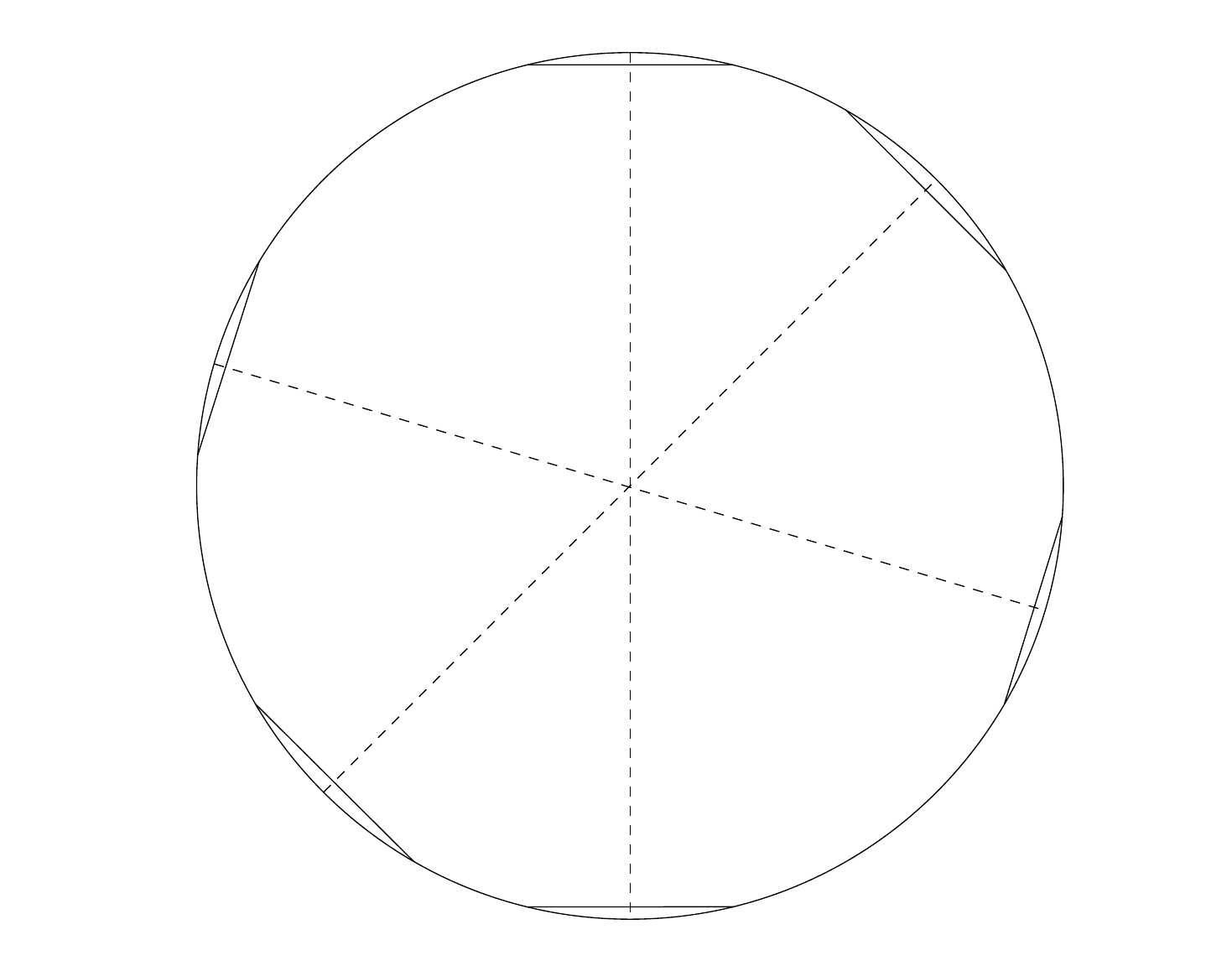} 
	\caption{Set of disjoint circular segments cut off from $B_\rho$ using chords perpendicular to directions $\pi/4$, $\pi/2$ and $7\pi/8$.}
	\label{figBEAM}
\end{figure}
and parallel to the sides of $\textup{Hull}(X)$, and these segments are all either disjoint or contained in each other for different
	rules, since $\rho$ is large. No additional infection takes place in $B_\rho$, therefore 
	$D=B_\rho\setminus\langle \Zdd\setminus B_\rho\rangle_\UU$
	is inwards stable. 
\end{proof}
Now, given $\rho>0$ we denote the discrete ball as \[B'_{\rho}:=\Zdd\cap B_{\rho}.\]
An immediate consequence of the above lemma is the fact that every vertex which is eventually
infected should be within some constant distance from an initially infected vertex.
\begin{coro}\label{dilation}
	If $\Ss(\UU)=S^1$, there exists $\hat \rho>0$ such that, for every $x\in \genA_\UU$,
	\begin{equation}
	  A\cap B'_{\hat \rho}(x)\ne \emptyset.  
	\end{equation}
\end{coro}
\begin{proof}
	Let $D$ be an inwards stable droplet with $0\in D$, and $\hat \rho>0$ such that
	$D\subset B_{\hat \rho}$. 
	Given $x\in \genA_\UU$, the translation $x+D$ is also inwards stable and
	$x\in \genA_{\UU}\cap (x+D)$. Thus \[A\cap B_{\hat \rho}(x)\supset A\cap  (x+D)\ne\emptyset.\]
\end{proof}

\begin{defi}[Dilation Radius]
	We define the {\it dilation radius} $\beta:=\beta(\UU)$ to be the smallest radius $\hat \rho\ge 1$
	satisfying the conclusion in Corollary \ref{dilation}. 
\end{defi}
Note that
\begin{equation}
  |B'_{3\beta}|\le 30\beta^2.  
\end{equation}

\subsection{Exponential decay}
We will use a specific collection of finite subtrees of $\Zdd$.
\begin{defi}
	For $n \ge 0$ we let $\mathcal T_{0,n}$ to be the collection of all trees $T\subset\Zdd$ containing the origin $0\in\Zdd$
	and other $n$ vertices (so that $|T|=n+1$). We also define the collection
	of all trees containing $0$ and having at most $n$ vertices ($|T|\le n$) by
	\begin{equation}
	  \mathcal T_{\le n}:= \bigcup_{k=1}^n \mathcal T_{0,k-1}.  
	\end{equation}
\end{defi}
A key ingredient to prove the exponential decay theorem is an upper bound for $|\mathcal T_{\le n}|$. 
The following proposition is a particular case of a beautiful problem in the book
{\it The art of mathematics: Coffee time in Memphis} (see Problem 45 in \cite{Btaofm}).
\begin{prop}\label{nofsubt}
	For every $n\ge 1$ we have
	$|\mathcal T_{0,n}|\le (3e)^n$. As a consequence, $|\mathcal T_{\le n}|\le (3e)^n$.
\end{prop}
Consider $\UU$-bootstrap percolation with initially infected set $A\subset\Zdd$, 
where $\Ss(\UU)=S^1$ and let $\beta$ be the dilation radius.
\begin{lemma}[Extremal lemma for $\mathcal K$]\label{extremaldeK}
	If $|\mathcal K|\ge n$ 
	then, there exists a tree $T\in\mathcal T_{\le n}$ such that
	\begin{equation}
	  |A\cap T|\ge (30\beta^2)^{-1} n. \end{equation}
\end{lemma}
\begin{proof}
In fact, let us suppose that $|\mathcal K|\ge 30\beta^2 n$,
and recursively find $n$ distinct vertices $x_1',\dots,x_n'\in A\cap T$, for
some tree $T\in\mathcal T_{\le 30\beta^2 n}$.

By definition of $\beta$, for $x_1=0\in\genA_\UU$ there exists $x_1'\in A\cap B'_\beta(x_1)$.
Then set $K_1= B'_{3\beta}(x_1)$, and since $|K_1|\le 30\beta^2$
we can find a vertex $x_2\in\mathcal K\setminus K_1$, 
which is at distance 1 from $K_1$;
now we apply Corollary \ref{dilation} to $x_2\in\genA_\UU$ and find a new vertex 
$x_2'\in A\cap B'_{\beta}(x_2)$.
Proceed in this way, for $i\le n$, assume we have found
vertex $x_i'\in A\cap B'_{\beta}(x_{i-1})$, then set
\[K_i=B'_{3\beta}(x_{i})\cup K_{i-1}.\]
Since $|K_i|\le 30\beta^2i$, for $i=1,\dots,n-1$ we have
\[|\mathcal K\setminus K_i|\ge 30\beta^2n-30\beta^2i\ge 1,\] so we can find a vertex $x_{i+1}\in\mathcal K\setminus K_i$,
which is at distance 1 from $K_i$. 
Observe that at step $n-1$ we still have $|\mathcal K\setminus K_{n-1}|\ge 30\beta^2\ge 1$, so for $x_n\in\mathcal K\setminus K_{n-1}$ we can apply
the corollary one more time to get our last vertex $x_n'\in A$.
For $i=1,\dots,n$, the vertices $x_i'$ are all distinct because all balls
$B'_\beta(x_i)$ are pairwise disjoint by construction.

Finally, consider a spanning tree $T$ of $K_n$, and note that $x_i,x_i'\in T$ for all
$i=1,\dots,n$. In particular, $|A\cap T|\ge n$, and the fact that 
$T\in\mathcal T_{\le 30\beta^2 n}$
follows from $0=x_1\in T$ and $|T|\le|K_n|\le 30\beta^2 n$.
\end{proof}

The same proof allows us to deduce another similar extremal lemma.
\begin{lemma}\label{stuck}
	There exists a constant $\lambda\in(0,30\beta^2]$ such that, if $\genA_{\UU}$ is connected then,
\begin{equation}
    |\genA_{\UU}|\le \lambda |A|.
\end{equation}
\end{lemma}
\begin{proof}
	If $A$ is infinite we have nothing to show. Assume $A$ is finite, then it is contained in a big
	rectangle $R\subset\Zdd$, since $\pm e_1,\pm e_2\in\Ss$, so $\genA_\UU\subset R$ is also finite.
	Since $\genA_{\UU}$ is connected, the above proof shows that $|\genA_\UU|>30\beta^2 n$ implies $|A|>n$.
	In other words, $|A|=n$ implies $|\genA_\UU|\le 30\beta^2 n= 30\beta^2 |A|$.
\end{proof}

The following is a quantitative reformulation of Theorem \ref{expdecIntro2}, whose proof is inspired by lines through the book {\it Percolation} of Bollob\'as and Riordan (see pp. 70 in \cite{ByR}).

\begin{theor}[Exponential decay for the cluster size]\label{expdecay}
Consider subcritical $\UU$-bootstrap percolation on $\Zdd$ with $\Ss(\UU)=S^1$ and let $\beta\ge 1$ be the dilation radius.
If $0<\varepsilon <e^{-150\beta^2}$ and $C=C(\varepsilon):=-\frac{1}{60\beta^2}\log(\varepsilon)$, then
	\begin{equation}
	\p_\varepsilon(|\mathcal K|\ge n)\le 
 \varepsilon^{\frac{1}{60\beta^2}n}=e^{-Cn},
	\end{equation}
	for every $n\in\mathbb N$.
\end{theor}
\begin{proof}
	By Lemma \ref{extremaldeK} and Proposition \ref{nofsubt}, with
	$\delta=(30\beta^2)^{-1}$, we obtain
	\begin{align*}
	\p_\varepsilon(|\mathcal K|\ge n) &
	\le  \p_\varepsilon\Bigg(\bigcup_{T\in\mathcal T_{\le n}} \{|A\cap T|\ge \delta n\}\Bigg)
 \le \sum_{T\in\mathcal T_{\le n}}\p_\varepsilon(|A\cap T|\ge \delta n) 
 \le \sum_{T\in\mathcal T_{\le n}}
	{{n}\choose{\delta n}}\varepsilon^{\delta n}    \\
& \le \sum_{T\in\mathcal T_{\le n}}  (e\delta^{-1}\varepsilon)^{\delta n}  
\le ([3e][e\delta^{-1} \varepsilon]^{\delta})^{n}   
\le e^{-Cn},
	\end{align*}
	and we are done.
\end{proof}


\section{Lower bounds via beams process}\label{SectionLowerBeams}\label{SectionLowerBeam}   
To deal with the cases $c\ge a+b$ we introduce a new tool which we 
call the {\it beams process}. This time, 
instead of covering the infected vertices step by step with components,
we cover them with beams, so that when we observe this induced process along the 
$e_3$-direction it looks like subcritical two-dimensional bootstrap percolation.

Consider the family $\mathcal N_{m}^{a,b}$ 
given by the collection of all subsets of size $m$ of
\begin{equation}\label{Nab0}
 N_{a,b}=\{a'e_1:\pm a'\in[a]\}\cup \{b'e_2:\pm b'\in[b]\}.
\end{equation}
Observe that $\Ss(\mathcal N_{m}^{a,b})=S^1$   if and only if $m\ge a+b+1$,
in particular, our exponential decay result (Theorem \ref{expdecay}) holds for these families.
From now on we set
\begin{equation}
m:=a+b+1. 
\end{equation}
\begin{defi}\label{beam}
A {\it beam} is a finite subset of $\Zdt$ of the form $H\times[w]$, where $H\subset\Zdd$ is 
 connected and $\langle H\rangle_{\mathcal N_{m}^{a,b}}=H$.
\end{defi}


It will be important for us to have an upper bound on the number of beams of a given size,
which are contained in $[L]^3$.
The following lemma is another consequence of Proposition \ref{nofsubt}.
\begin{lemma}[Counting beams]\label{nofbeams}
	Let $\mathcal B_{n_1,n_2}$ be the collection of all copies of the beam $H\times[w]$ contained in $[L]^3$ satisfying $w\le n_1$ and $|H|\le n_2$.
	Then
	\[|\mathcal B_{n_1,n_2}|\le n_1L^3(3e)^{n_2}.\]
\end{lemma}
\begin{proof}
The number of segments inside $[L]$ 	with at most $n_1$ vertices, is at most $n_1L$.

	Now we give an upper bound for the number of $H$'s. Let $\mathcal H_h$ denote the collection of all
	connected sets $H\subset[L]^2$ such that $|H|=h$, 
	so we can write \[h|\mathcal H_h|=\sum_{H\in\mathcal H_h}|H|=\sum_{x\in[L]^2}\sum_{H\in\mathcal H_h}
	\mathds 1\{x\in H\}= \sum_{x\in[L]^2} \textup{cs}(x),\]
	where $\textup{cs}(x)$ is the number of connected subsets of $[L]^2$ with size $h+1$, containing a fixed point $x$.
	To each of such sets we can associate an spanning tree in an injective fashion, so by Proposition
	\ref{nofsubt}, $|\mathcal H_h|\le L^2(3e)^{h-1}$. It follows that the number of $H$'s is at most
	\[\sum_{h=1}^{n_2} |\mathcal H_h|  \le L^2\sum_{h=1}^{n_2} (3e)^{h-1} \le L^2(3e)^{n_2}.\]
\end{proof}

\subsection{The beams process}
\begin{defi}\label{genbeam}
Given finite connected sets $S_1,S_2\subset\Zdt$, we say that a beam $H\times[w]$ is {\it generated by} $(S_1,S_2)$ if it can be constructed in the following way: by translating $S_1\cup S_2$ if necessary, we can assume that the smallest rectangle containing it is $R\times[w]$, 
 then consider the connected sets $H_1,H_2\subset \Zdd$ given by
 \[H_t:=\{x\in R: (\{x\}\times[w])\cap S_t\ne\emptyset\},\ \ \ t=1,2.\]    
If $\langle H_1\cup H_2\rangle_{\mathcal N_{m}^{a,b}}$ is connected then we take $H:=\langle H_1\cup H_2\rangle_{\mathcal N_{m}^{a,b}}$.
Otherwise, we let $P\subset R$ be any path with minimal size connecting $H_1$ to $H_2$ and then set $H:=\langle H_1\cup H_2\cup P\rangle_{\mathcal N_{m}^{a,b}}$.
\end{defi}
In this definition $\langle S_1\cup S_2\rangle\subset H\times[w]$ for each $r\ge m$, and generated beams could depend on the choice of the path $P$. However, such minimal paths are not relevant for our purposes. 
\begin{nota}
 We will denote any fixed beam generated by $(S_1,S_2)$ as $B(S_1\cup S_2)$, regardless the choice of $P$.
\end{nota}
We want to track the process of infection by covering all possible infected
sites with beams, we do that step by step in order to get some control over the sizes.
The following algorithm is a variation of the components process.
We will use it to show an Aizenman-Lebowitz-type lemma which says that when $[L]^3$
is internally filled, then it contains {\it covered} beams of all intermediate sizes (see Lemma \ref{ALbeams} below).
\begin{defi}[The beams process]\label{beampro}
 Let $A=\{x_1,\dots,x_{|A|}\}\subset [L]^3$ and fix $r\ge c+1$.
 Set $\mathcal B:=\{S_1,\dots,S_{|A|}\}$, where $S_i=\{x_i\}$ for each $i=1,\dots,|A|$, and
repeat until STOP:
 \begin{enumerate}
  \item If there exist distinct beams $S_{1},S_{2}\in\mathcal B$ such that
  \[S_{1}\cup S_{2}\]
  is strongly connected, then remove it from $\mathcal B$,
  and replace by $B(S_{1}\cup S_{2})$. 
  \item If there do not exist such a family of sets in $\mathcal B$, then STOP.
 \end{enumerate}
We call any beam $S=B(S_{1}\cup S_{2})\subset [L]^3$ added to the collection $\mathcal B$
a {\it covered beam}, and denote the event that $S$ is covered by
$I^{\text{\ding{86}}}(S).$
\end{defi}
Again, there are two properties that are due to the way the algorithm evolves:
\begin{itemize}
 \item  Any covered beam $S$ satisfies $\langle A\cap S\rangle\subset \langle S\rangle=S$.
 \item 
 The process stops in finite time, thus, we can consider the final collection $\mathcal B'$ and set  $V(\mathcal B'):=\bigcup\limits_{S\in\mathcal B'}S$. 
By using the same arguments in the proof of Lemma \ref{final=gen}, it follows that $\genA\subset V(\mathcal B')$. 
\end{itemize}


\subsection{Case $c=a+b$}\label{sectionbeam}
In this section we prove the following.
\begin{prop}\label{lower3}
Under $\mathcal N_{m}^{a,b,a+b}$-bootstrap percolation, there is a constant $\gamma=\gamma(a,b)>0$
such that, if 
\[L<\exp(\gamma p^{-1}),\]
then 
$\p_p[I^{\bullet}([L]^3)]\to 0,\ as\ p\to 0.$
\end{prop}
The beams process and Lemma \ref{stuck} allow us to prove a beams version of the Aizenman-Lebowitz Lemma for this case.
Let $\lambda>0$ be the constant in Lemma \ref{stuck}
associated
to the subcritical two-dimensional family $\mathcal N_{m}^{a,b}$.
\begin{lemma}\label{ALbeams}
	Consider $\mathcal N_{r}^{a,b,c}$-bootstrap percolation with $r\ge c+1$.
 If $[L]^3$ is internally filled, then there is a constant $\kappa\in\mathbb N$ such that 
 for every $k=\kappa,\dots, L$, there exists a covered beam $H\times[w]$ satisfying $w,|H|\le 3\lambda k$, and either $w\ge k$ or $|H|\ge k$.
\end{lemma}
\begin{proof}
	Let $S=H\times[w]$ be the first beam that appears in the beam process satisfying either $w\ge k$ or $|H|\ge k$
	(such a set exists since $V(\mathcal B')=[L]^3$).
	Then, it is enough to show that $w\le rk$ and $|H|\le 3\lambda k$.
	
	We know that $S=B(S_{1}\cup S_{2})$ for some beams $S_{t}=H_t\times[w_t]$ such that $S_{1}\cup S_{2}$ is strongly connected. Moreover,
	by definition of $S$, $w_t\le k-1$ for $t=1,2$, so
	\begin{equation}\label{maxdiam}
	w\le 2 \max_{t=1,2}\{w_t\} +r \le r(k-1)+r\le rk.
	\end{equation}
	Analogously, $|H_t|\le k-1$, and we know that $H= \langle H_1\cup H_2\cup P\rangle_{\mathcal N_{m}^{a,b}}$ for some path $P$ with bounded (or zero) size, $H$ is connected, so by Lemma \ref{stuck},
	\begin{equation}\label{maxdiam}
	|H|\le \lambda\cdot 2 \max_{t=1,2}\{|H_t|\} + \lambda|P|\le 2\lambda(k-1) + O(1)\le 3\lambda k.
	\end{equation}
\end{proof}
Now, let us prove the lower bound in the case $c=a+b$.
\begin{proof}[Proof of Proposition \ref{lower3}]
Take $L<\exp(\gamma p^{-1})$, where $\gamma>0$ is some small constant. Let us show that
$\p_p(I^{\bullet}([L]^3))$ goes to $0$, as $p\to 0$. Fix $\varepsilon>0$.

If $[L]^3$ is internally filled, by Lemma \ref{ALbeams} there exists a covered beam $S=H\times[w]\subset [L]^3$
satisfying $w,|H|\le\varepsilon/p$, and moreover, either $w\ge\varepsilon/3\lambda p$ or $|H|\ge \varepsilon/3\lambda p$, hence, by union bound, $\p_p[I^{\bullet}([L]^3)]$ is at most
\begin{align*}
 \sum_{S\in\mathcal B_{\frac{\varepsilon}{p},\frac{\varepsilon}{p}}}\big(\p_p[I^{\text{\ding{86}}}(S)\cap
\{w\ge\varepsilon/3\lambda p\}]+\p_p[I^{\text{\ding{86}}}(S)\cap \{|H|\ge \varepsilon/3\lambda p\}]\big).
\end{align*}

To bound the first term, we use the fact that $H\times[w]$ is covered; this implies that there is no gap
of size $r$ along the $e_3$-direction.
Therefore, by considering the $w/r$ disjoint slabs,
if $\varepsilon$ is small, then there exists some $c_1=c_1(\varepsilon,r)>0$ such that
\begin{align*}
\p_p[I^{\text{\ding{86}}}(H\times[w])\cap
\{w\ge\varepsilon/3\lambda p\}] & \le
\Big(1-e^{-\Omega(p|H|)}\Big)^{w/r} 
 =\Big(1-e^{-\Omega(\varepsilon)}\Big)^{\varepsilon/2r\lambda p}
 \le e^{-c_1/p}.
\end{align*}

To bound the second term, for each $S\in\mathcal B_{\frac{\varepsilon}{p},\frac{\varepsilon}{p}}$ consider the set
\[A':=\left\{x\in[L]^2: (\{x\}\times[w])\cap\langle A\cap S\rangle \ne\emptyset
\right\}.\]
In other words, $x\in A'$ if and only if there exists $y_1\in \{x\}\times [w]$ such that 
either $y_1\in A$, or $y_1\in S$ got infected by using at least $m$
infected neighbours in $y+N_{a,b}$, where
$N_{a,b}$ is given by (\ref{Nab0}).
Now, by applying Markov's inequality,
\[\p_p(A\cap(\{x\}\times [w])\ne \emptyset) = 
O(wp)\le \varepsilon.\]
Therefore, by monotonicity we can couple the process on $[L]^2\times[w]$ having initial infected set
$A$, with $\mathcal N_{m}^{a,b}$-bootstrap
percolation on $[L]^2\times\{1\}\subset\Zdd$ where the initial infected set is chosen to be $\varepsilon$-random.

In particular, under $\mathcal N_{m}^{a,b}$-bootstrap percolation there should exist a connected component of size at least
$|H|\ge \varepsilon/3\lambda p$ inside $[L]^2$. 
On the other hand, there are at most
$L^2$ possible ways to place the origin in $H$, so if $\mathcal K$ denotes the cluster of 0, Theorem \ref{expdecay} implies

\begin{align*}
\p_p[I^{\text{\ding{86}}}(S)\cap \{|H|\ge \varepsilon/3\lambda p\}]
& \le \sum_{x\in [L]^2}\p_\varepsilon(\{|\mathcal K|>\varepsilon/3\lambda p\}\cap \{x=0\}) 
 \le L^{2}\p_\varepsilon(|\mathcal K|\ge \varepsilon/3\lambda p) \\
& \le  e^{2\gamma/p}e^{-C\varepsilon/3\lambda p} 
 = e^{-(C\varepsilon/3\lambda -2\gamma)/p},  
\end{align*}
where $C=-\frac{1}{60\beta^2}\log\varepsilon$ and we choose $\varepsilon>0$ such that $C\varepsilon>0$ and $\gamma<C\varepsilon/6\lambda$ at first.
By Lemma \ref{nofbeams} we conclude that
\begin{align*}
\p_p[I^{\bullet}([L]^3)] & \le \sum_{S\in\mathcal B_{\frac{\varepsilon}{p},\frac{\varepsilon}{p}}}\big( e^{-c_1/p}+e^{-(C\varepsilon/3\lambda -2\gamma)/p}  \big) 
 \le \frac{\varepsilon}{p}L^3(3e)^{\varepsilon/p}e^{-c_2/p}\\ 
& \le e^{4\gamma/p}e^{\varepsilon\log(3e)/p}e^{-c_2/p}\to 0,
\end{align*}
for $c_2,\gamma>0$ small enough.
\end{proof}

\subsection{The coarse beams process}
  In this section we study the last case $c\ge a+b+1$. The lower bound will be proved by using a coupling with subcritical two-dimensional bootstrap percolation again, as we did in the previous section, however, this time we infect squares instead of single vertices. 
The trick now is to consider the following coarser process.
\begin{defi}[Coarse bootstrap percolation]\label{coarse}
 Assume that $b+1$ divides $L$ and we partition $[L]^2$ as $L^2/(b+1)^2$ copies of
 $\boxplus:=[b+1]^2$ in the obvious way, and think of $\boxplus$ as a single vertex in the new scaled grid $[L/(b+1)]^2$. 
 Given a two-dimensional family $\UU$, suppose we have some fully infected copies of $\boxplus\in [L/(b+1)]^2$ and denote this initially infected set by $A$, then we define
 {\it coarse} $\UU$-bootstrap percolation to be the result of applying $\UU$-bootstrap percolation
 to the new rescaled vertices. We denote the closure of this process by $\genA_{b}$.
\end{defi}
To avoid trivialities, we assume that $b+1$ divides $L$. Set
\[m:=a+b+1< c+1 = r.\]
\begin{defi}\label{coarbeam}
A {\it coarse beam} is a finite set 
of the form $H\times[w]$,
 where $H\subset\Zdd$ is connected and $\langle H\rangle_{b}=H$
 under coarse $\mathcal N_{m}^{a,b}$-bootstrap percolation.
\end{defi}
\begin{nota}\label{constru2}
Given finite connected sets $S_1,S_2\subset [L]^2\times[L]$, we partition $[L]^2$ as in Definition \ref{coarse} and denote by $B_b(S_1\cup S_2)$ the coarse beam generated by $(S_1,S_2)$ which is constructed in the (coarse) analogous way, as we did in Definition \ref{genbeam}, using coarse paths when needed. Note that every coarse beam is a beam in the sense of the previous section.
\end{nota}
The following algorithm is a refinement of that one given in Definition \ref{beampro}.
\begin{defi}[The coarse beams process]
Let $A=\{x_1,\dots,x_{|A|}\}\subset [L]^3$ and fix $r\ge c+1$.
 Set $\mathcal B:=\{S_1,\dots,S_{|A|}\}$, where $S_i=\{x_i\}$ for each $i=1,\dots,|A|$, and
repeat until STOP:
 \begin{enumerate}
  \item If there exist distinct beams $S_{1},S_{2}\in\mathcal B$ such that
  \[S_{1}\cup S_{2}\]
  is strongly connected, and $\langle S_{1}\cup S_{2} \rangle\ne 
  S_{1}\cup  S_{2}$,
  then choose a minimal such family, remove it from $\mathcal B$,
  and replace by a coarse beam $B_b(S_{1}\cup S_{2})$. 
  \item If there do not exist such a family of sets in $\mathcal B$, then STOP.
 \end{enumerate}
We call any beam $S=B_b(S_{1}\cup  S_{2})$ added to the collection $\mathcal B$ a {\it covered beam}, and denote the event that $S$ is covered by
$I_b^{\text{\ding{86}}}(S).$
 \end{defi}

 The two highlighted usual properties are preserved for this algorithm too:
\begin{itemize}
 \item  Any covered beam $S$ satisfies $\langle A\cap S\rangle\subset \langle S\rangle=S$.
 \item There is a final collection $\mathcal B'$ and we can set
 $V(\mathcal B'):=\bigcup\limits_{S\in\mathcal B'}S$. Then, we also have $\genA\subset V(\mathcal B')$.
\end{itemize}
\subsection{Case $c>a+b$}
  In this section we prove the lower bound corresponding to the last case. 
\begin{prop}\label{lower4}
	Under $\mathcal N_{c+1}^{a,b,c}$-bootstrap percolation with $c>a+b$, there exists a constant $\gamma=\gamma(c)>0$
	such that, if 
	\[L<\exp(\gamma p^{-1}(\log p)^2),\]
then 
$\p_p[I^{\bullet}([L]^3)]\to 0,\ as\ p\to 0.$
\end{prop}
We state an analogue of Lemma \ref{ALbeams} for the coarse beams setting without proof
because the arguments are exactly the same. However, we obtain slightly different constants since the number of vertices of the form $\boxplus$ in a coarse beam $H$ equals
$|H|/(b+1)^2$.

Consider $\mathcal N_{r}^{a,b,c}$-bootstrap percolation with $r\ge c+1$, and let $\kappa_0$ be a large positive integer depending on $a,b,c$ and $r$.
\begin{lemma}\label{ALlema2}
 If $[L]^3$ is internally filled then for every $h,k=\kappa_0,\dots,L$, there exists
 a covered (coarse) beam
 $H\times[w]\subset [L]^3$ satisfying $w\le rk$, $|H|\le 2(b+1)^2\lambda h$, and
 either $w\ge k$ or $|H|\ge h$.
\end{lemma}
%

Finally, we prove the lower bound for the remaining case.
\begin{proof}[Proof of Propositon \ref{lower4}]
Take $L<\exp(\gamma p^{-1}(\log p)^2)$, where $\gamma>0$ is some small constant. Let us show that
$\p_p(I^{\bullet}([L]^3))$ goes to $0$, as $p\to 0$. Fix $\delta>0$ and set
\[h=\delta p^{-1}\log \tfrac 1p ,\ \ \ k=p^{-\frac 32}\]
If $[L]^3$ is internally filled, by Lemma \ref{ALlema2} there exists a covered beam
$S=H\times[w]\subset [L]^3$ satisfying $w\le k$, $|H|\le (b+1)^2h$, and
either $w\ge k/2\lambda$ or $|H|\ge h/2\lambda$ (as we said, the cardinality of
$H$ viewing $S$ as a beam equal $(b+1)^2|H|$ viewing $S$ as a coarse beam), hence $\p_p[I^{\bullet}([L]^3)]$ is at most
\begin{align*}
 \sum_{S\in\mathcal B_{k,(b+1)^2h}}\Big(\p_p[I_b^{\text{\ding{86}}}(S)\cap
\{w\ge k/2\lambda\}]+\p_p[I_b^{\text{\ding{86}}}(S)\cap \{|H|\ge h/2\lambda\}]\Big).
\end{align*}
To bound the first term, we use the fact that 
$A\cap (H\times\{rk+1,\cdots,rk+r\})\ne \emptyset$ for all $k=0,\dots,w/r-1$, 
since $H\times[w]$ is covered. Therefore, for some $c_1>0$,
\begin{align*}
\p_p[I_b^{\text{\ding{86}}}(H\times[w])\cap  \{w\ge k/2\lambda\}] 
&\le \big(1-(1-p)^{rh}\big)^{w/r}
 \le \big(1-e^{-2r\varepsilon\log\frac 1p}\big)^{k/2r\lambda} 
 = \big(1-p^{2r\varepsilon}\big)^{k/2r\lambda}\\
& \le e^{-p^{2r\varepsilon-\frac 32}/2r\lambda}
=e^{-c_1 p^{-1}(\log\frac 1p)^2}.  
\end{align*}
To bound the second term we use the fact that $r=c+1\ge a+b+2$. More precisely, if $[L]^3$ is internally filled, then every copy of
$[b+1]^2\times [L]$
should contain at least 2 vertices of $A$ within some constant distance,
otherwise, there is no way to infect such a copy.

Then, given $S=H\times[w]\in\mathcal B_{k,(b+1)^2h}$ consider the set
$A'$ consisting of all copies of $\boxplus\subset [L]^2$ (as in Definition \ref{coarse})
such that the rectangle $\boxplus\times [w]\subset S$ contains at least 2 vertices of $A$ within
distance $r$. 
By union bound, the probability of finding such vertices is at most
\[\sum_{x\in\boxplus\times [w]}\sum_{0<\|y-x\|\le r}\p_p(x,y\in A)\le
\tilde{C}wp^2\le p^{\frac 13}.\]
Therefore, by monotonicity we can couple the process in $[L]^2\times[w]$ having initial infected set
$A$, with coarse $\mathcal N_{m}^{a,b}$-bootstrap
percolation on $[L/(b+1)]^2$ and initial infected set $\varepsilon$-random
with $\varepsilon=\varepsilon(p):=p^{1/3}$. 

In particular, under $\mathcal N_{m}^{a,b}$ (coarse) there should exist a connected component of size at least
$|H|\ge h/2\lambda$ inside $[L]^2$. Once more, there are at most
$L^2$ possible ways to place the origin in $H$, so if $\mathcal K$ denotes the (coarse) cluster of 0, Theorem \ref{expdecay} implies

\begin{align*}
\p_p[I_b^{\text{\ding{86}}}(S)\cap \{|H|\ge h/2\lambda\}]
& \le \sum_{\boxplus\subset [L]^2}\p_\varepsilon(\{|\mathcal K|>h/2\lambda\}\cap \{\boxplus=0\}) 
 \le L^{2}\p_\varepsilon(|\mathcal K|\ge h/2\lambda) \\
& \le  e^{2\gamma p^{-1}(\log\frac 1p)^2}e^{-Ch/2\lambda}
 = e^{-(c' -2\gamma) p^{-1}(\log\frac 1p)^2},  
\end{align*}
for some constant $c'=c'(\beta,\lambda)>0$ (recall that $C\sim-\log p$ asymptotically, by 
Theorem \ref{expdecay}). Take $\gamma<c'/2$ at first;
by Lemma \ref{nofbeams} we conclude that
\begin{align*}
\p_p[I^{\bullet}([L]^3)] & \le \sum_{S\in\mathcal 
	B_{k,(b+1)^2h}}\Big( e^{-c_1 p^{-1}(\log\frac 1p)^2}
+e^{-(c' -2\gamma)p^{-1}(\log\frac 1p)^2}  \Big) \\
&  \le kL^3(3e)^{(b+1)^2h}e^{-c_3p^{-1}(\log\frac 1p)^2}
 \le e^{4\gamma p^{-1}(\log\frac 1p)^2}e^{-c_3p^{-1}(\log\frac 1p)^2}\to 0,
\end{align*}
as $p\to 0$, for $c_3,\gamma>0$ small enough, and we are finished.
\end{proof}

\section{Future work}
All proofs in this paper extend to the case $r=c+2$, and can be used to determine $\log L_c\left(\mathcal N_{c+2}^{a,b,c},p\right)$ up to a constant factor for all triples $(a,b,c)$, except for $c=a+b-1$ which is a new interesting case to be studied separately.

In general, a problem which remains open is the determination of the threshold for other values of $r$.
We believe that the techniques used to prove Theorem~\ref{target0} can be adapted to cover all $c+1<r\le b+c$ 
(though significant technical obstacles remain); in this case, by Proposition \ref{genuppbound} below, the critical length is singly exponential.
However, to deal with the cases $b+c<r<a+b+c$,
the techniques required are likely to be more similar to those of~\cite{CC99} and~\cite{AA12}, and the critical length should be doubly exponential.

Finally, Theorem \ref{expdecay} can be generalized to all dimensions $d\ge 3$ and all families $\UU$ such that $\Ss(\UU)= S^{d-1}$.
However, we do not know if this property holds for subcritical families $\UU$ satisfying $\Ss(\UU)\ne S^{d-1}$.
In order to determine the critical lengths for general critical models, it could be useful to extend this result to a wider class of subcritical families.
\begin{problem}\label{probExDecay}
Characterize the subcritical $d$-dimensional update families $\UU$
such that $\mathcal K$ has the exponential decay property.
\end{problem}

\appendix
\section{A general upper bound for $r\le b+c$}
In this appendix we assume that $r\le b+c$ and show that the critical length is at most singly exponential in this case, as we claimed above.
Consider $\mathcal N_{r}^{a,b,c}$-bootstrap percolation.
\begin{prop}\label{genuppbound}
Given $r\in\{c+1,\dots,c+b\}$,
there exists $\Gamma=\Gamma(c)>0$ such that, if $L> L_c(\mathcal N_{r}^{b,c},p)^\Gamma$,
then $\pp_p\left(\langle A\rangle_{\mathcal N_{r}^{a,b,c}} = [L]^3\right)\to 1$, as $p\to 0$.
Thus, 
\[\log L_c\left(\mathcal N_{r}^{a,b,c},p\right)=O\left(\log L_c(\mathcal N_{r}^{b,c},p)\right)
=O\left(p^{-(r-c)}(\log p)^2\right).\]
\end{prop} 
\begin{rema}\label{coveredcases}
This proposition, in particular, already gives us the upper bound in the case $c>a+b$ of our main Theorem \ref{target0}.
It also shows that $\mathcal N_{r}^{a,b,c}$ is 2-critical for all $r\in\{c+1,\dots,c+b\}$;
in fact, since $L_c(\mathcal N_{r}^{a,b,c},p)$ is increasing in $r$, 
by Proposition \ref{genuppbound},
\[\log L_c\left(\mathcal N_{r}^{a,b,c},p\right)\le \log L_c\left(\mathcal N_{c+b}^{a,b,c},p\right)\le O\left(p^{-b}(\log \tfrac 1p)^2\right).\]
Moreover, by Theorem \ref{target0} we also have
\[\log L_c\left(\mathcal N_{r}^{a,b,c},p\right)\ge \log L_c\left(\mathcal N_{c+1}^{a,b,c},p\right)\ge \Omega\left(p^{-1/2}\right).\]
\end{rema}
To prove this proposition, we will use dimensional reduction by means of a renormalization argument, and show that filling the whole of $[L]^3$ is at most as hard as filling $L$ disjoint copies of $[L]^2$ which are orthogonal to the $e_1$-direction.

To do so in this regime, we will compare the family $\mathcal N_{r}^{a,b,c}$ with the two-dimensional
family $\mathcal N_{r}^{b,c}$ consisting of all subsets of size $r$ of the set
$N_{b,c}$ given by (\ref{Nab0}).
It turns out that $\mathcal N_{r}^{b,c}$ is critical if and only if $r$ belongs to this regime,
and in this case $\Ss(\mathcal N_{r}^{b,c})=\{\pm e_1,\pm e_2\}$.
The key step is to refine the upper bounds in (\ref{paso1}), by using standard renormalization techniques.
\begin{lemma}[Renormalization]\label{elactual}
Under $\mathcal N_{r}^{b,c}$-bootstrap percolation with $r\in\{c+1,\dots,c+b\}$,
there exists a constant $\Gamma'>0$ depending on $c$ such that,
	\begin{equation}\label{1menosL2}
	\pp_p\left(\langle A\rangle_{\mathcal N_{r}^{b,c}} = [L]^2\right)\ge
	1-\exp\left(-L^{1/2}\right),
	\end{equation}
	for all $p$ small enough and 
	$L> L_c(\mathcal N_{r}^{b,c},p)^{\Gamma'}$.
\end{lemma}
\begin{proof}
 See, e.g. \cite{Sch92}.
\end{proof}
Now, we prove the general upper bound.

\begin{proof}[Proof of Proposition \ref{genuppbound}]
Decompose $[L]^3$ as $L$ consecutive copies of $[L]^2$ all of them orthogonal to the $e_1$-direction, and
call those copies $R_i:=\{i\}\times[L]^2$.

Now, we couple the original process with the reduced two-dimensional processes;
	if for each $i\in\{1,\dots,L\}$,
	$\langle A\cap R_i\rangle_{\mathcal N_{r}^{b,c}}=R_i$ in the 
	$\mathcal N_{r}^{b,c}$-bootstrap process,
	then $[L]^3$ is internally filled. Therefore, by Lemma \ref{elactual} we have
	\begin{align*}  
	\pp_p\left(\langle A\rangle_{\mathcal N_{r}^{a,b,c}} = [L]^3\right)& \ge \pp_p\bigg(\bigcap\limits_{i=1}^L\{\langle A\cap R_i\rangle_{\mathcal N_{r}^{b,c}}=R_i\}\bigg)
	=\prod _{i=1}^L\pp_p\left(\langle A\cap R_i\rangle_{\mathcal N_{r}^{b,c}}=R_i\right)\\
	& \ge \left[1-\exp\left(-L^{1/2}\right)\right]^L\xrightarrow[p\to 0]{} 1,
	\end{align*}
	if $L> 
	\exp\left(\Gamma'\, p^{-(r-c)}(\log p)^{2\cdot\mathds 1_{\{c>b\}}}\right)$.
\end{proof} 
\section*{Acknowledgements}
The author would like to thank Rob Morris for introducing him to bootstrap percolation, and for his many invaluable suggestions. The author is very grateful to the  Instituto Nacional de Matem\'atica Pura e Aplicada (IMPA) for the time and space to create, research and write in this strong academic environment.




\bibliographystyle{plain}
\bibliography{References}
\end{document}